\documentclass[letterpaper, 10pt, conference]{ieeeconf}
\pdfoutput=1\IEEEoverridecommandlockouts 
\usepackage{amsmath,amssymb,url}
\usepackage{graphicx,subfigure,enumerate}
\usepackage[pdfpagemode=none,pdfstartview=FitH]{hyperref}

\newcommand{\al}{\alpha}
\newcommand{\be}{\beta}
\newcommand{\ep}{\epsilon}
\newcommand{\lam}{\lambda}

\newcommand{\diag}{\text{diag}}

\newcommand{\I}{I_{3\times3}}
\renewcommand{\t}{\times}

\newcommand{\tr}{\text{tr}}
\newcommand{\no}{\nonumber}

\newcommand{\SO}{\ensuremath{\mathsf{SO(3)}}}
\newcommand{\so}{\ensuremath{\mathfrak{so}(3)}}

\newcommand{\T}{\ensuremath{\mathsf{T}}}

\newcommand{\refeqn}[1]{(\ref{eq:#1})}
\renewcommand{\Re}{\ensuremath{\mathbb{R}}}
\newcommand{\w}{{\omega}}
\newcommand{\W}{{\Omega}}
\newcommand{\U}{\mathcal{U}}
\newcommand{\V}{\mathcal{V}}

\title{\LARGE \bf
Angular Velocity Observer on the Special Orthogonal Group for Velocity-Free Rigid-Body Attitude Tracking Control}

\author{Tse-Huai Wu and Taeyoung Lee\authorrefmark{1}%
\thanks{Tse-Huai Wu and Taeyoung Lee, Mechanical and Aerospace Engineering, The George Washington University, Washington DC 20052. {\tt \{wu52,tylee\}@gwu.edu}}%
\thanks{This research has been supported in part by NSF under the grant CMMI-1243000 (transferred from 1029551), CMMI-1335008, and CNS-1337722.}
}

\newtheorem{prop}{Proposition}

\begin{document}
\allowdisplaybreaks
\maketitle \thispagestyle{empty} \pagestyle{empty}

\begin{abstract}
This paper studies a rigid body attitude tracking control problem with attitude measurements only, when angular velocity measurements are not available. An angular velocity observer is constructed such that the estimated angular velocity is guaranteed to converge to the true angular velocity asymptotically from almost all initial estimates. As it is developed directly on the special orthogonal group, it completely avoids singularities, complexities, or discontinuities caused by minimal attitude representations or quaternions. Then, the presented observer is integrated with a proportional-derivative attitude tracking controller to show a separation type property, where exponential stability is guaranteed for the combined observer and attitude control system.
\end{abstract}

\section{Introduction}

The problem of attitude control of a rigid body is one of the most popular research topics in control theory and practice. The corresponding applications include aerial and underwater vehicles, robotics, and spacecraft dynamics. Many approaches have been studied in the attitude control problem to address various technical challenges~\cite{Tayebi08,ChaSanICSM11,LeeCST13,IEPC2013}. In most of the attitude control strategies, full states measurements, i.e., both attitude and angular velocity measurements, are required. However, angular velocity measurements are not available in certain cases, for example, due to limited sensing, power availability, and costs. 

Several approaches have been proposed for attitude controls without angular velocity measurements, where the value of the angular velocity is estimated. A nonlinear angular velocity observer is presented by Salcudean in~\cite{Sal1991}, to construct an estimated angular velocity in terms of the attitude measurements based on an observer designed for a second-order linear system. However, the observer is designed and analyzed separately from attitude control systems, assuming that there is a separation principle-like property, i.e., it is assumed that the convergence of the controller is independent of the observer design. Recently, a switching-type angular velocity observer is presented to show stability of an attitude control system in terms of quaternions~\cite{Chu2014}. 

There are other attitude control techniques that do not require an estimate of the angular velocity. An auxiliary system approach is proposed based on the passivity property~\cite{Liz1996,Tsiotras98}, where the auxiliary system generates a damping term similar to a derivative term that is directly dependent on the angular velocity~\cite{Tayebi08}.  Additionally, a hybrid attitude tracking controller is proposed in the absence of angular velocity information~\cite{Sch2012}, and a velocity-free adaptive controller is developed for rigid-body attitude tracking~\cite{Costic2000}.

Most of these prior works on angular velocity observers and velocity-free attitude controls are constructed in terms of local parameterizations of the attitudes, or quaternions. Attitude control systems based on minimal representations, such as Euler angles or modified Rodrigues parameters, suffer from singularities in representing large angle rotational maneuvers. Quaternions do not have singularities. However, since the configuration space of quaternions, represented by three-sphere double-covers the attitude configuration space of the special orthogonal group, one physical attitude actually corresponds to two antipodal quaternions. This ambiguity should be carefully resolved for any quaternion-based attitude control system to avoid undesirable phenomena such as unwinding, where a rigid body unnecessarily rotates through a large angle, even if the initial attitude error is small, or it may become sensitive to small measurement noise~\cite{BhaBerSCL00,MaySanITAC11}.

This paper follows the first type of approaches that are based on an estimated value of the angular velocity. An angular velocity observer is constructed directly on the special orthogonal group, and it is shown that the zero equilibrium of the estimation errors are almost globally asymptotically stabile, i.e., it is asymptotically stable and the region of attraction only excludes a set of zero Lebesgue measure~\cite{ChaSanICSM11}. The second part of this paper is devoted to a separation type property by integrating the proposed angular velocity observer with a separately designed attitude tracking control system. It is shown that the combined system yields exponential stability.

Compared with the prior work~\cite{Sal1991,Chu2014}, the angular velocity observer presented in this paper completely avoids the aforementioned issues of quaternions. Furthermore, in the switching-based angular velocity observer~\cite{Chu2014}, the observer performance depends on the mass distribution of the rigid body, since the convergence rate becomes slower and the number of switching increases as the rigid body becomes more elongated. Frequent switching may cause undesired behaviors or even instability as illustrated by numerical examples presented later in this paper. The main contribution of this paper can be summarized as (i) developing an angular velocity observer on the special orthogonal group to avoid the issues of quaternions and the dependency of convergence rates on the shape of a rigid body, and (ii) showing a separation property mathematically rigorously and explicitly without need for discontinuities caused by switching. To author's best knowledge, a separation-type property of attitude controls and angular velocity estimation has not been studied before without a switching logic.

%The observed angular velocity converges to true angular velocity asymptotically. Expressly, the estimate error dynamics is showed to be almost globally asymptotically stable by rigorous Lyapunov analysis. We then show the separation-type property by combining the observer with a separately designed attitude controller to show the closed-loop system is exponentially stable. Compared with~\cite{Sal1991, Chu2014}, the observer presented in this paper is in a simple manner such that sign function or switching law is not required. By modeling the estimate system directly on $\SO$ without any parameterizations, the complexity of the observer design is significantly reduced.  

The paper is organized as follows. A rigid-body dynamic model is introduced at Section II. An angular velocity observer is presented at Section \ref{sec:obs}, and a separation-type property is shown at Section IV, followed by numerical examples at Section \ref{sec:num}.

\section{Rigid Body Attitude Dynamics}\label{sec:Dynamics}
Consider the attitude dynamics of a fully-actuated rigid body. Two coordinate frames are defined: an inertial reference frame and a body-fixed frame. The attitude of the rigid body is denoted by $R\in\SO$ that represents the transformation of a representation of a vector from the body-fixed frame to the inertial reference frame. The configuration manifold of attitude is the special orthogonal group:
	\begin{align*} 
	\SO=\{R\in\Re^{3\t3}\,|\,{R^\T}R=I,\,\det[R]=1\}.
	\end{align*}
Let $\w\in\Re^3$ and $\W\in\Re^3$ denote the angular velocity of the rigid body with respect to the inertial reference frame and the body-fixed frame, respectively.	
The governing equations for the rigid body attitude dynamics are given by
	\begin{gather} 
	\frac{d}{dt}(J\w)=\tau,~~ J=RJ_0R^\T, \label{eq:eom1}\\
	\dot{R}=\hat{\w}R=R\hat{\W}, \label{eq:eom2}
	\end{gather} 
where $J_0\in\Re^{3\t3}$ is the fixed inertia matrix expressed in body-fixed frame and $\tau$ is the control moment expressed in the inertial reference frame. Note that the equation of motion for the angular velocity, \refeqn{eom1} is represented with respect to the inertial frame. In addition, the \textit{hat} map $\wedge:\Re^3\rightarrow\so$ transforms a vector in $\Re^3$ to a $3\t3$ skew-symmetric matrix such that $\hat{x}y=x^\wedge y=x\t y$ for any $x,\,y\in\Re^3$. And the inverse of hat map is denoted by the \textit{vee} map $\vee:\so\rightarrow\Re^3$. Several properties of hat map are listed as follows:
	\begin{gather} 
	\tr[A\hat{x}]=\tr[\hat{x}A]=-x^\T(A-A^\T), \label{eq:hat1}\\
	R\hat{x}R^\T =(Rx)^\wedge, \label{eq:hat2}\\
	\hat{x}A+A^\T\hat{x}=(\{\tr[A]\I-A\}x)^\wedge \label{eq:hat3},
	\end{gather}
for any $x\in\Re^3,\,A\in\Re^{3\t3},\,R\in\SO$.  The standard inner product of two vectors is denoted by $x\cdot y=x^\T y$. Throughout this paper, $\I$ denotes the $3\t3$ identity matrix and the 2-norm of matrix $A$ is denoted by $\|A\|$. Also, $\lam_M$ and $\lam_m$ are defined as the maximum eigenvalue and the minimum eigenvalue of the inertia matrix $J_0$, respectively.

\section{Angular Velocity Observer on $\SO$} \label{sec:obs}

In this section, an observer is constructed such that the angular velocity is estimated when the attitude measurements and the control input are available.

\subsection{Estimate Frame}

Define an orthonormal frame estimated by the observer. The attitude and angular velocity of the estimate frame with respect to the inertial reference frame are denoted by $\bar{R}\in\SO$ and $\bar\w\in\Re^3$, respectively. More explicitly, $\bar R$ denotes the linear transformation from the inertial reference frame to the estimate frame. The discrepancy between the true attitude $R$ and the estimated attitude $\bar{R}$ is denoted by a rotation matrix $Q_E\in\SO$, where    
	\begin{align} 
	Q_E=R\bar{R}^\T.
	\end{align}
Note that $Q_E=\I$ when $R=\bar{R}$. 

To further describe the error dynamics between $R$ and $\bar{R}$, the estimate error variables are defined as follows:
	\begin{gather} 
	\Psi_E=\frac{1}{2}\tr[G_E(\I-Q_E)], \label{eq:ev1}\\
	e_{R_E} =\frac{1}{2}(Q_EG_E-G_EQ_E^\T)^\vee, \label{eq:ev2}\\
	e_{\w_E} =J\w-J\bar\w, \label{eq:ev3}
	\end{gather}
where $\Psi_E\in\Re$, $e_{R_E}\in\Re^3$ and $e_{\w_E}\in\Re^3$ denote the estimate error function, attitude estimate error vector and estimate angular velocity error, respectively. The matrix $G_E$ is defined as $G_E=\diag[\ep_1,\ep_2,\ep_3]\in\Re^{3\t3}$ where $\ep_1,\ep_2,\ep_3\in\Re$ are distinct positive constants.

\subsection{Observer Design}
The observer dynamics are defined as
	\begin{gather} 
	\frac{d}{dt}(J\bar\w)=\tau+\frac{1}{2}k_EJ^{-1}e_{R_E}, \label{eq:ob1}\\
	\dot{\bar R}=\big[Q_E^\T(\bar\w+k_vJ^{-1}e_{R_E})\big]^\wedge\,\bar{R}, \label{eq:ob2}
	\end{gather}
where $k_E,\,k_v\in\Re$ are positive constants. The observer is designed in the inertial reference, and it can be transformed to the body-fixed frame easily since the attitude is assumed to be available.	

The estimate error variables along the solution of the above observer dynamics satisfy the following properties.
\begin{prop} \label{prop:erdyn}
The estimate error variables $Q_E$, $\Psi_E$, $e_{R_E}$, and $e_{\w_E}$ satisfy:
	\begin{enumerate}[(i)]
		\item $\Psi_E$ is positive definite about $R=\bar{R}$.
		\item Let the positive constants $n_1,\ldots,n_5$ be
				\begin{align} 
		n_1 &= \mathrm{min}\{\ep_1+\ep_2,\, \ep_2+\ep_3,\, \ep_3+\ep_1\}, \label{eq:n1}\\
		n_2 &= \mathrm{max}\{(\ep_1-\ep_2)^2,\, (\ep_2-\ep_3)^2,\, (\ep_3-\ep_1)^2\}, \nonumber\\
		n_3 &= \mathrm{max}\{(\ep_1+\ep_2)^2,\,(\ep_2+\ep_3)^2,\,(\ep_3+\ep_1)^2\}, \nonumber\\
		n_4 &= \mathrm{max}\{\ep_1+\ep_2,\, \ep_2+\ep_3,\, \ep_3+\ep_1\}, \nonumber\\
		n_5 &= \mathrm{min}\{(\ep_1+\ep_2)^2,\,(\ep_3+\ep_3)^2,\,(\ep_3+\ep_1)^2\},\nonumber
		\end{align}		
		and let $\psi_E<n_1$. The error function $\Psi_E$ is locally quadratic, i.e.,
			\begin{align} 
			\hspace*{-0.1cm}\frac{n_1}{n_2+n_3}\|e_{R_E}\|^2 \leq\Psi_E\leq \frac{n_1n_4}{n_5(n_1-\psi_E)}\|e_{R_E}\|^2, \label{eq:PsiE}
			\end{align}
		where the upper bound is satisfied when $\Psi_E<\psi_E$.
		\item $\dot{Q}_E=\hat{\w}_E Q_E$,
		\item $\dot{\Psi}_E=\w_E^\T e_{R_E}$,
		\item $\dot{e}_{R_E}=E_o(R,\bar{R})\w_E$,
		\item $\dot{e}_{\w_E}=-\frac{1}{2}k_EJ^{-1}e_{R_E}$,
	\end{enumerate}	
%	\begin{gather} 
%	\dot{Q}_E=\hat{\w}_E Q_E, \label{eq:ev1}\\
%	\dot{\Psi}_E=\w_E^\T e_{R_E}, \label{eq:ev2}\\
%	\dot{e}_{R_E}=E(R,\bar{R})\w_E, \label{eq:ev3}
%	\end{gather}
where $\w_E\in\Re^3$ and $E_o(R,\bar{R})\in\Re^{3\t3}$ are given by
	\begin{gather} 
	\w_E=\w-\bar{\w}-k_vJ^{-1}e_{R_E},  \label{eq:wE}\\
	E_o(R,\bar{R})=\frac{1}{2}(\tr[Q_E]\I-2\hat{e}_{R_E}-G_EQ_E^\T).
	\end{gather}
\end{prop}

\begin{proof} %==================================
It is known that $-1\leq\tr[R]\leq3$, for any rotation matrix $R\in\SO$, then it is clear that $\Psi_E\geq0$ and $\Psi_E=0$ only happens at $Q_E=\I$, which verifies (i). To show (ii), the following properties in~\cite{Fernando11} are applied: For non-negative constants $f_1,f_2,f_3$, let $F=\text{diag}[f_1,f_2,f_3]\in\Re^{3\times 3}$, and let $P\in\SO$. Define
	\begin{gather} 
	\Phi=\frac{1}{2}\tr[F(\I-P)], \label{eq:Phi}\\
	e_P=\frac{1}{2}(FP-P^{\T}F)^\vee.\label{eq:eP}
	\end{gather}
Then, $\Phi$ is bounded by the square of the norm of $e_P$ as
	\begin{gather}
	\frac{h_1}{h_2+h_3}\|e_P\|^2\leq \Phi \leq\frac{h_1h_4}{h_5(h_1-\phi)}\|e_P\|^2. \label{eq:PHI}
	\end{gather}
If $\Phi<\phi<h_1$ for a constant $\phi$, where $h_i$ are given by
	\begin{align*} 
	h_1 &= \mathrm{min}\{f_1+f_2,\, f_2+f_3,\, f_3+f_1\}, \\
	h_2 &= \mathrm{max}\{(f_1-f_2)^2,\, (f_2-f_3)^2,\, (f_3-f_1)^2\}, \\
	h_3 &= \mathrm{max}\{(f_1+f_2)^2,\, (f_2+f_3)^2,\, (f_3+f_1)^2\}, \\
	h_4 &= \mathrm{max}\{f_1+f_2,\, f_2+f_3,\, f_3+f_1\}, \\
	h_5 &= \mathrm{min}\{(f_1+f_2)^2,\,(f_2+f_3)^2,\,(f_3+f_1)^2\}.
	\end{align*}
Now, we choose $F=G_E$ and $P=Q_E$, we then have $\Phi=\Psi_E$, $e_P=e_{R_E}$, $\phi=\psi_E$ and $h_i=n_i$, for $i=\{1,2,\ldots,5\}$. This shows (ii).
 
From \refeqn{eom2} and \refeqn{ob2}, the time-derivative of $Q_E$ is 
	\begin{align*} 	
	\dot{Q}_E 
	&=\hat{\w}R\bar{R}^\T-R\bar{R}^\T\big[Q_E^\T(\bar{\w}+k_vJ^{-1}e_{R_E})\big]^\wedge.
	\end{align*}
Using \refeqn{hat2} and \refeqn{wE}, it is rearranged as
	\begin{align*}
	\dot{Q}_E 	&= \hat{\w}Q_E-Q_EQ_E^\T(\bar{\w}+k_vJ^{-1}e_{R_E})^\wedge Q_E \\ &=(\w-\bar{\w}-k_vJ^{-1}e_{R_E})^\wedge Q_E = \hat{\w}_EQ_E, 
	\end{align*}
which shows (iii). 

Next, the time-derivative of $\Psi_E$ is 
	\begin{align*} 
	\dot{\Psi}_E
	&= -\frac{1}{2}\tr[\dot{Q}_EG_E] \\
	&=-\frac{1}{2}\tr\big[\big(\w-\bar{\w}-k_vJ^{-1}e_{R_E}\big)^\wedge \,Q_EG_E\big].
	\end{align*}
From \refeqn{hat1}, it is rewritten as
	\begin{align*}
	\dot{\Psi}_E
	&= \frac{1}{2}(\w-\bar{\w}-k_vJ^{-1}e_{R_E})^\T(Q_EG_E-G_EQ_E^\T)^\vee \\
	&=(\w-\bar{\w}-k_vJ^{-1}e_{R_E})^\T e_{R_E} \triangleq \w_E^\T e_{R_E},
%	&= e_{R_E}^\T (J^{-1}e_{\w_E}) -k_ve_{R_E}^\T (J^{-1}e_{R_E}),
	\end{align*}
which shows (iv). Next, according to \refeqn{hat3} and \refeqn{ev2}, the time-derivative of $e_{R_E}$ is 
	\begin{align*} 
	\dot{e}_{R_E} 
	&= \frac{1}{2}\big(\hat{\w}_EQ_EG_E +G_EQ_E^\T\hat{\w}_E\big)^\vee \\
	&= \frac{1}{2}(\tr[Q_EG_E]\I-Q_EG_E)\w_E \\
	&= \frac{1}{2}(\tr[Q_E]\I-2\hat{e}_{R_E}-G_EQ_E^\T)\w_E, \no
	\end{align*}
which shows (v). Finally, from \refeqn{eom1}, \refeqn{ev3} and \refeqn{ob1}, we have  
	\begin{align*} 
	\dot{e}_{\w_E} =\frac{d}{dt}(J\w)-\frac{d}{dt}(J\bar\w)=\tau-(\tau+\frac{1}{2}k_EJ^{-1}e_{R_E}),
	\end{align*}
which shows (vi).
\end{proof}	%==================================	

Next, we show that the zero equilibrium of the estimate error variables is almost globally asymptotically stable.
\begin{prop} \label{prop:observer}
Consider the system given by \refeqn{eom1}, \refeqn{eom2} with the angular velocity observer given by \refeqn{ob1}, \refeqn{ob2}. The following properties holds:
	\begin{enumerate}[(i)]
		\item There are four equilibrium configurations, given by
			\begin{align} 
			(R,\,\w)\in\{(\bar{R},\bar{\w}),\,(D_i\bar{R},\,\bar{\w})\},
			\end{align}					
		for $i=1,2,3$, where $D_1=\diag[1,-1,-1]$,  $D_2=\diag[-1,1,-1]$ and  $D_3=\diag[-1,-1,1]$.
		\item The desired equilibrium $(R,\w)=(\bar{R},\bar{\w})$ is almost globally asymptotically stable.
		\item The remaining three undesired equilibrium configurations are unstable.
	\end{enumerate}	
\end{prop}

\begin{proof} %==================================
The equilibrium configurations happen at $(e_{R_E}, e_{\w_E})=(0,0)$. Clearly, in view of \refeqn{ev3}, $e_{\w_E}=0$ yields $\w=\bar\w$. From \refeqn{ev2}, $e_{R_E}=0$ directly implies that $Q_EG_E-G_EQ_E^\T=0$, which follows that either $Q_E=\I$ or $\tr[Q_E]=-1$~\cite[Theorem 5.1]{Mahony08}. Therefore,
	\begin{align*} 
	Q_E=R\bar{R}^\T 
	\in&\{I_{3\times 3},\, D_1,\, D_2,\, D_3\},
	\end{align*}
which shows (i). 

Consider the following Lyapunov function,
	\begin{align} 
	\U=e_{\w_E}^\T e_{\w_E}+k_E\Psi_E, \label{eq:U}
	\end{align}
which is positive definite about $(e_{\w_E},e_{R_E})=(0,0)$. From the properties (iv) and (vi) of Proposition \ref{prop:erdyn}, the time-derivative of $\U$ is given by 		
	\begin{align} 
	\dot{\U}
	&=2e_{\w_E}^\T\dot{e}_{\w_E} +k_E\dot{\Psi}_E \no\\
	&=2e_{\w_E}^\T(-\frac{1}{2}k_EJ^{-1}e_{R_E})+(\w-\bar{\w}-k_vJ^{-1}e_{R_E})^\T e_{R_E} \no\\
	&=-k_Ek_ve_{R_E}^\T (J^{-1}e_{R_E}) \leq -k_Ek_v\frac{1}{\lam_M}\|e_{R_E}\|^2. \label{eq:Udot}
	\end{align}
Hence, one can conclude that $e_{\w_E}$, $e_{R_E}$ are globally bounded and $\lim_{t\rightarrow\infty}\|e_{R_E}\|=0$. Further, one can show that $\|\ddot{Q}_E\|$ is bounded and $\lim_{t\rightarrow\infty}\int^t_0\|\dot{Q}_E\|\,dt'=\lim_{t\rightarrow\infty}\|Q_E\|$ exits. From Barbalat Lemma, we conclude that
	\begin{gather} 
	\lim_{t\rightarrow\infty}\|\dot{Q}_E\|=\lim_{t\rightarrow\infty}\|\w-\bar{\w}-k_vJ^{-1}e_{R_E}\|=0.
	\end{gather}		
Since $\lim_{t\rightarrow\infty}\|e_{R_E}\|=0$, it is clear that $\lim_{t\rightarrow\infty}\|\w-\bar\w\|=0$, and this implies $\lim_{t\rightarrow\infty}\|e_{\w_E}\|=0$. Consequently, the equilibrium $(e_{R_E}, e_{\w_E})=(0,0)$ is asymptotically stable. 

However, the fact that $\lim_{t\rightarrow\infty}e_{R_E}=0$ does not necessarily imply that the estimated attitude asymptotically converges to the true attitude. Instead, it asymptotically converges to either the true attitude or one of the three undesired equilibria given by $RD_i$ for $1\leq i\leq 3$. 

Next, we show the undesired equilibria are unstable.  At the first undesired equilibrium $\bar R=RD_1$, we have $\Psi_E=\epsilon_2+\epsilon_3$. Define $\mathcal{W}=k_E(\epsilon_2+\epsilon_3)-\mathcal{U}$, or
\begin{align*}
\mathcal{W} = k_E (\epsilon_2+\epsilon_3-\Psi_E) -\|e_{\omega_E}\|^2.
\end{align*}
Then, $\mathcal{W}=0$ at the undesired equilibrium. Due to the continuity of $\Psi_E$, in an arbitrarily small neighborhood of $RD_1$ in $\SO$, there exists $\bar R$ such that $(\epsilon_2+\epsilon_3)-\Psi_E > 0 $. For such attitudes, we can guarantee that $\mathcal{W} > 0$ if $\|e_{\omega_E}\|$ is sufficiently small. In other words, at any arbitrarily small neighborhood of the undesired equilibrium, there exists a domain, namely $U$ such that $\mathcal{W} >0$ in $U$. And we have $\dot{\mathcal{W}}= -\dot{\mathcal{U}} > 0$ from \refeqn{Udot}. According to Theorem 3.3 at~\cite{Kha96}, the undesired equilibrium is unstable. The instability of the other two equilibrium configurations can be shown by the similar way. This shows the almost global asymptotic stability of (ii) as well as (iii).
%
%(iii) 
%
%
%
%, due to the existence of three additional equilibria. Instead, we show that those undesired equilibria are unstable. More explicitly, at the three undesired equilibria, the value of the Lyapunov function is $\U=2k_E$. Let $\mW=2k_E-\U$, i.e.,
%	\begin{align*} 
%	\mW=k_E(2-\Psi_E)-(\|e_{\w_E}\|^2+k_E\Psi_E).
%	\end{align*}
%At the three undesired equilibria $\mW=0$, and we have	
%	
%We can choose $R$ arbitrary such that $k_E(2-\Psi_E)>0$. Consequently, if $\|e_{\w_E}\|$ is sufficiently small, we can obtain $\mW>0$. Therefore, at any arbitrarily small neighborhood of the undesired equilibrium, there exists a set in which $\mW>0$ and $\dot{\mW}=-\dot{\U}>0$. We can conclude that these three equilibria are unstable~[Theorem3.3]\cite{Kha96}. And both of (ii) and (iii) are verified.
\end{proof} %==================================

The presented angular velocity observer guarantees that the estimation errors asymptotically converge to zero for almost all initial estimates, i.e., the region of attraction excludes only a \textit{thin} set of zero measure. This is the strongest stability property for any continuous angular velocity observer, due to the topological restriction stating that it is impossible to achieve global attractivity in the special orthogonal group unless discontinuities are introduced. 

In contrast to the prior work given by~\cite{Chu2014} where the the ratio of $\frac{\lambda_M}{\lambda_m}$ has a crucial impact on the observer performance, the convergence property of the proposed angular velocity observer is independent of the mass distribution of the rigid body.

\section{Attitude Tracking without Angular Velocity Measurements}

In this section, we show a separation property of the angular velocity observer developed in this previous section with a proportional-derivative attitude tracking control system on $\SO$.

\subsection{Attitude Tracking Controls}

We first review a attitude tracking controller developed on $\SO$~\cite[Sec. 11.4.3]{BulLew05} and~\cite{LeeCST13}. %, and this controller will be combined with the angular velocity observer later.
%The governing equations are expressed in the body-fixed frame
%	\begin{gather} 
%	J_0\dot{\W}+\W\t J_0\W=u,\label{eq:eom3}\\ 
%	\dot{R}=R\hat{\W},\label{eq:eom4}
%	\end{gather}
%where $u,\,\W$ are the control moment and the angular velocity expressed in the body-fixed frame. Note that \refeqn{eom3}, \refeqn{eom4} are equivalent to \refeqn{eom1},\refeqn{eom2}, respectively. 
%
Suppose the desired attitude $R_d(t)\in\SO$ and the desired angular velocity $\W_d(t)\in\Re^3$ are given as smooth functions of time, and they satisfy the kinematic equation $\dot{R}_d=R_d\hat{\W}_d$. Let $Q\in\SO$ be the relative attitude of the desired attitude with respect to the current attitude of the rigid body, i.e.,
	\begin{align*} 
	Q=R^\T R_d\in\SO,
	\end{align*}	
The attitude tracking error variables are defined as 
		\begin{gather} 
		\Psi=\frac{1}{2}\tr[G(\I-Q)],\label{eq:Psi}\\
		e_R=\frac{1}{2}(GQ^\T-QG)^\vee,\\
		 e_\W=\W-Q\W_d,
		\end{gather}
where $\Psi\in\Re$ is the tracking attitude error function; $e_R,\,e_\W\in\Re^3$ are the tracking attitude error vector and tracking angular velocity error, respectively. The matrix $G=\diag[g_1,g_2,g_3]\in\Re^{3\t3}$ where $g_1,g_2,g_3\in\Re$ are distinct positive constants. 

The corresponding error dynamics are given as
	\begin{gather} 
	\dot{\Psi}=e_R^\T e_\W, \label{eq:ed1}\\
	\dot{e}_R=E_c(Q)e_\W, \label{eq:ed2}\\
	J_0\dot{e}_\W=u+\hat{\chi}e_\W-JQ\dot{\W}_d-\widehat{Q\W_d}J_0Q\W_d, \label{eq:ed3}
	\end{gather}
where $E_c(Q)\in\Re^{3\times 3}$ and $\chi\in\Re^3$ are defined as
	\begin{gather} 
	E_c(Q)=\frac{1}{2}(\tr[QG]\I-QG), \\
	\chi=J_0e_\W+(2J_0-\tr[J_0]\I)Q\W_d,\label{eq:chi}
	\end{gather}	
and $u\in\Re^3$ is the control moment expressed in the body-fixed frame, i.e., $u=R^\T\tau$. Detailed analysis of the error variables has been addressed in~\cite{LeeCST13,IEPC2013,BulLew05}.

% There are four equilibria for $(e_R, e_\W)=(0,0)$, and the desired equlibrium lies in  $(R,\W)=(R_d,\W_d)$.

A proportional-derivative (PD) type controllers on $\SO$ is introduced as below.
\begin{prop}{(\cite{LeeCST13,IEPC2013,BulLew05})}
Consider the attitude dynamics given by \refeqn{eom1}, \refeqn{eom2}. For positive constants $k_R,\,k_\W\in\Re$, let the control input be
	\begin{align} 
	u=-k_Re_R-k_\W{e_\W} +J_0Q\dot\W_d +\widehat{Q\W_d}(J_0Q\W_d). \label{eq:full}
	\end{align}
Then, the zero equilibrium of the tracking errors $(e_R,e_\Omega)$ is almost globally asymptotically stable.
\end{prop}

\subsection{Separation-Type Property} \label{sec:sep}

The above PD-type attitude tracking control system requires that the angular velocity of the rigid body is available always. Here we show that the angular velocity observer presented at Section \ref{sec:obs} satisfies a separation property when combined with the PD-type controller.  

Suppose that the true angular velocity $\Omega$ is not available to the control system, and the angular velocity estimated by the presented observer is applied instead. Define the estimated angular velocity tracking error as 
	\begin{align}  
	\bar{e}_\W=\bar{\W}-Q\W_d,\label{eq:eWbar}
	\end{align}	
where $\bar{\W}=R^\T\bar{\w}$ is the estimate angular velocity expressed in the body-fixed frame. The stability properties of the corresponding combined observer and controller are summarized as follows.

\begin{prop} \label{prop:sep}
Consider the attitude dynamics given by \refeqn{eom1}, \refeqn{eom2} with the angular velocity observer given by \refeqn{ob1}, \refeqn{ob2}. The control input is chosen as
	\begin{gather} 
	u=-k_Re_R-k_\W{\bar{e}_\W} +J_0Q\dot\W_d +\widehat{Q\W_d}(J_0Q\W_d), \label{eq:free}
	\end{gather}
for positive constants $k_R,\,k_\W\in\Re$. 
Assume that the inertia matrix $J_0$ of the rigid body and the weighting matrix $G_E$ satisfy 
			\begin{gather} 
	        \frac{\lam_M}{\lam_m}<\frac{\tr[G_E]}{\|G_E\|}. \label{eq:ratio}
			\end{gather}		
Let $\bar\psi_E$ be a positive constant satisfying $\bar\psi_E < \min\{n_1,\frac{1}{2}(\tr[G_E]-\frac{\lam_M}{\lam_m}\|G_E\|)\}$. Also, assume that the initial conditions satisfy
			\begin{gather} 
			{\Psi}_E(0) <\bar\psi_E < \min\{n_1,\frac{1}{2}(\tr[G_E]-\frac{\lam_M}{\lam_m}\|G_E\|)\}, \label{eq:roa1}\\
			\|e_{\w_E}(0)\|^2<  k_E(\psi_E-\Psi_E(0)).\label{eq:roa2}
			\end{gather}
		Then, the desired equilibrium given by $(R,\Omega,\bar R,\bar\Omega)=(R_d,\Omega_d,R,\Omega)$ is exponentially stable. 

\end{prop}
\begin{proof} %==================================
See Appendix.
\end{proof} %==================================

This proposition implies a separation property that the presented angular velocity observer guarantees exponential stability even when combined with an attitude tracking control system. Unlike~\cite{Chu2014} where a switching logic, that may cause frequent switchings, is introduced, the trajectories along the presented velocity-free attitude control scheme is free of discontinuities. This is critical in practice, as illustrated by numerical examples in the next section.

While the performance of the angular velocity observer presented in the previous section is independent of the inertia matrix, the ratio of the maximum eigenvalue to the minimum eigenvalue of the inertia matrix should satisfy \refeqn{ratio} for the separation property when combined with the attitude tracking control system. Most of existing spacecraft satisfy the assumption \refeqn{ratio}, and several numerical studies show that the separation property still holds even for various elongated rigid bodies that do not satisfy \refeqn{ratio}. Extension of the presented results to eliminate \refeqn{ratio} is referred to as future investigation.

\section{Numerical Examples} \label{sec:num}

To illustrate the performance of the presented angular velocity observer, we consider two cases for attitude stabilization and attitude tracking. 

\subsection{Attitude Stabilization}

We first consider a case of detumbling a rigid body, where the desired attitude and angular velocity are given by $R_d=\I$ and $\W_d=0$. The inertial matrix is given by $J_0=\diag[5,1,2]\,\mathrm{kgm^2}$. The initial condition is specified as $R(0)=\text{exp}(\pi/4\hat{e}_1)$ and $\W(0)=[1,\,-1.5,\,2.5]\,\mathrm{rad/sec}$ where $e_1=[1~0~0]^\T$. The matrix $G$ and $G_E$ are selected to be $G=G_E=\text{diag}[1.1,\, 1,\, 0.9]$. In particular, the control gains are given by $k_R=16J_0$, $k_\W=k_v=5.6J_0$, and $k_E=10J_0$. Note that the controller gains $k_R, k_\W, k_v$ are given to be scalars throughout this paper but they can be easily generalized to symmetric positive definite matrices. Numerical result of the proposed observer is illustrated at Fig.\ref{fig:SO3-1}, which exhibits excellent convergence properties. 

For a comparison, the angular velocity observer and controller presented~\cite{Chu2014} are applied as well, and the corresponding numerical results are illustrated at Fig. \ref{fig:switch-1}. As it is developed in terms of quaternions, the attitude estimation error and the attitude tracking error are plotted as the scalar part and the vector part of quaternions. It is shown that there are frequent switchings when $t\leq 3$ seconds, and the corresponding control input has high-frequency oscillations. 

\begin{figure}[h]
\centerline{
	\hspace*{0.015\columnwidth}
	\subfigure[Attitude estimation error $e_{R_E}$]{\includegraphics[width=0.5\columnwidth]{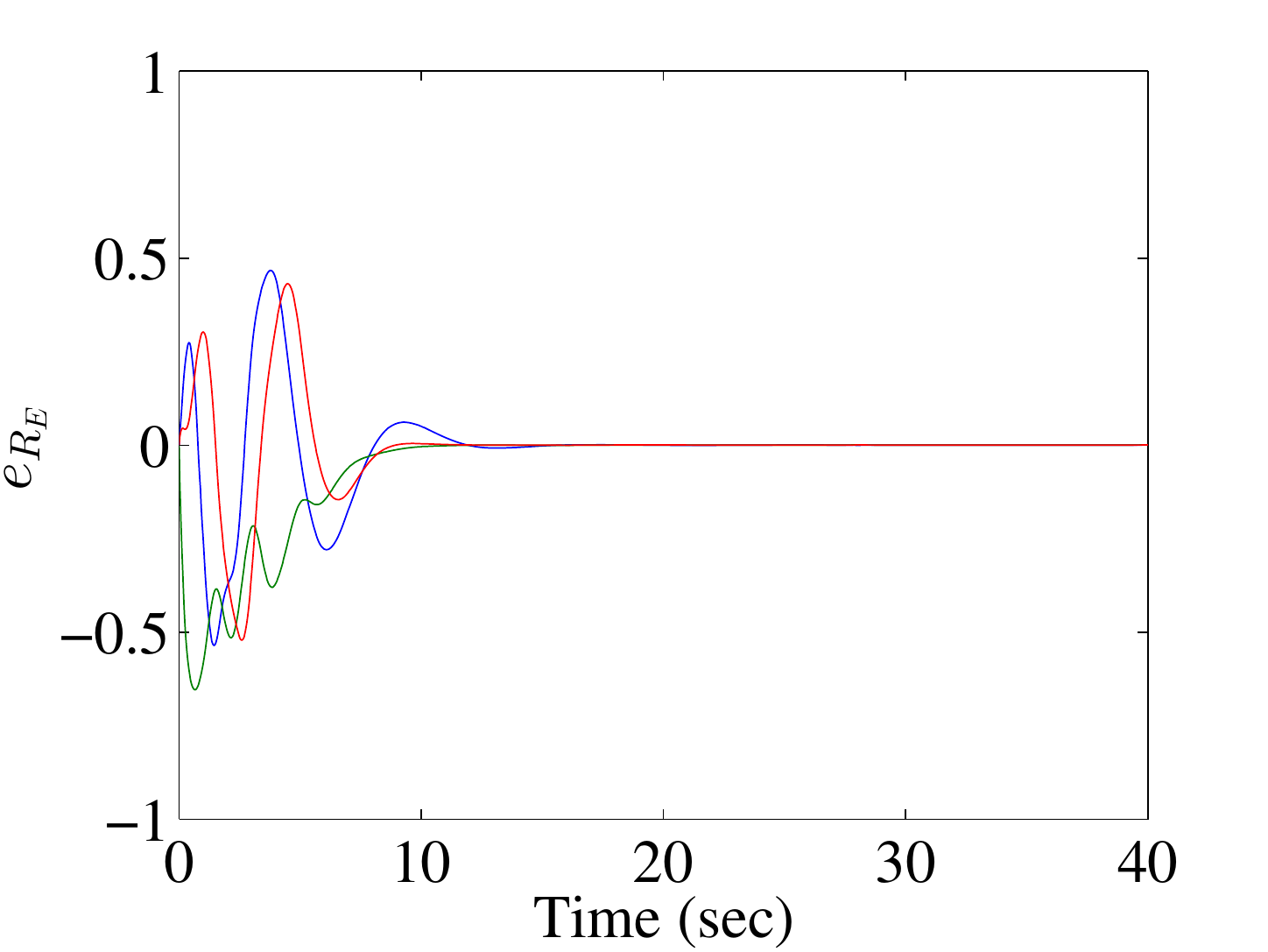}}
	\hspace*{-0.015\columnwidth}
	\subfigure[Angular velocity estimation error $\w-\bar{\w}$]	
{\includegraphics[width=0.5\columnwidth]{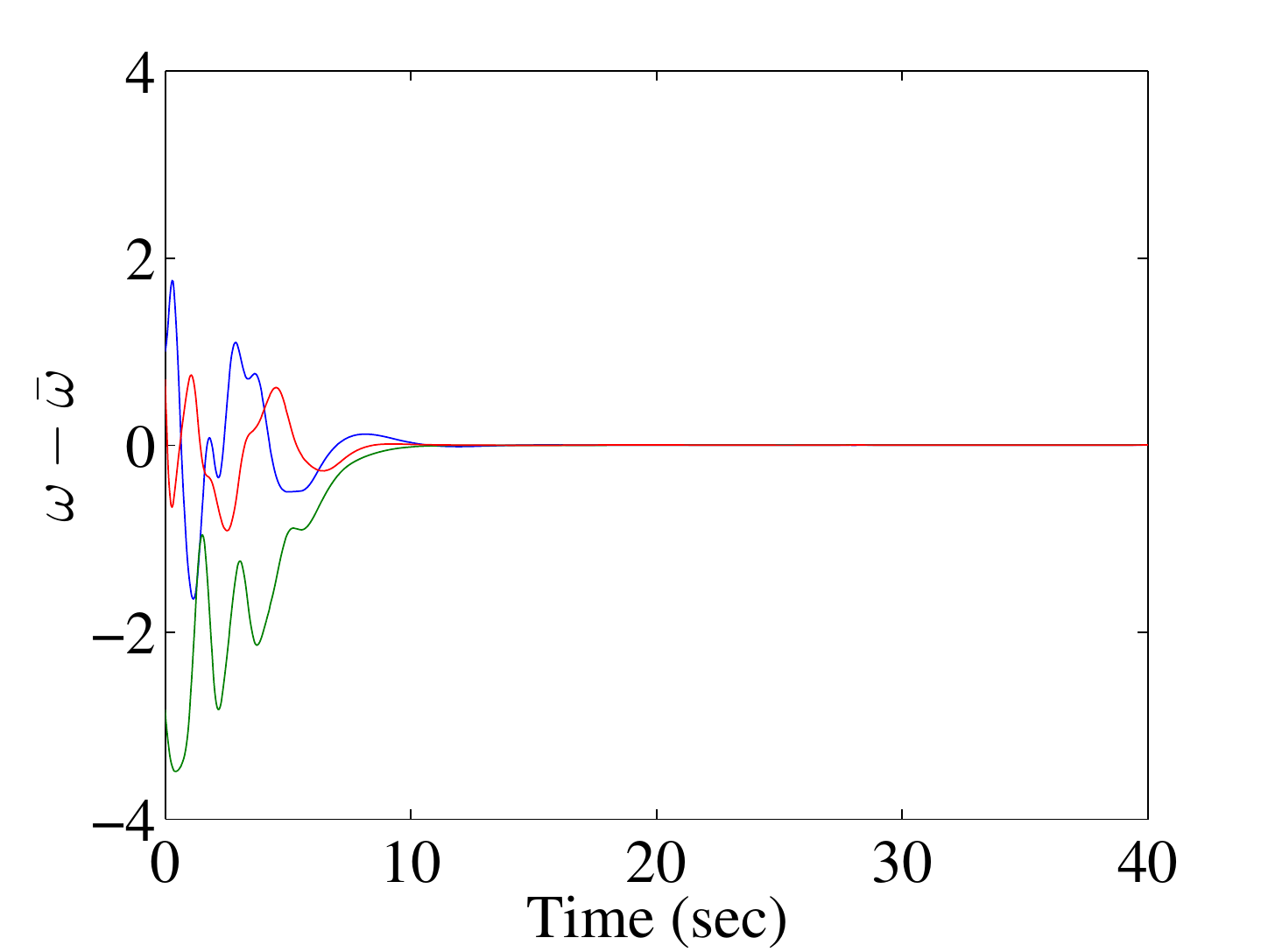}}
}	
\vspace*{-0.02\columnwidth}
\centerline{
	\hspace*{0.015\columnwidth}
	\subfigure[Attitude stabilization error $e_R$]{\includegraphics[width=0.5\columnwidth]{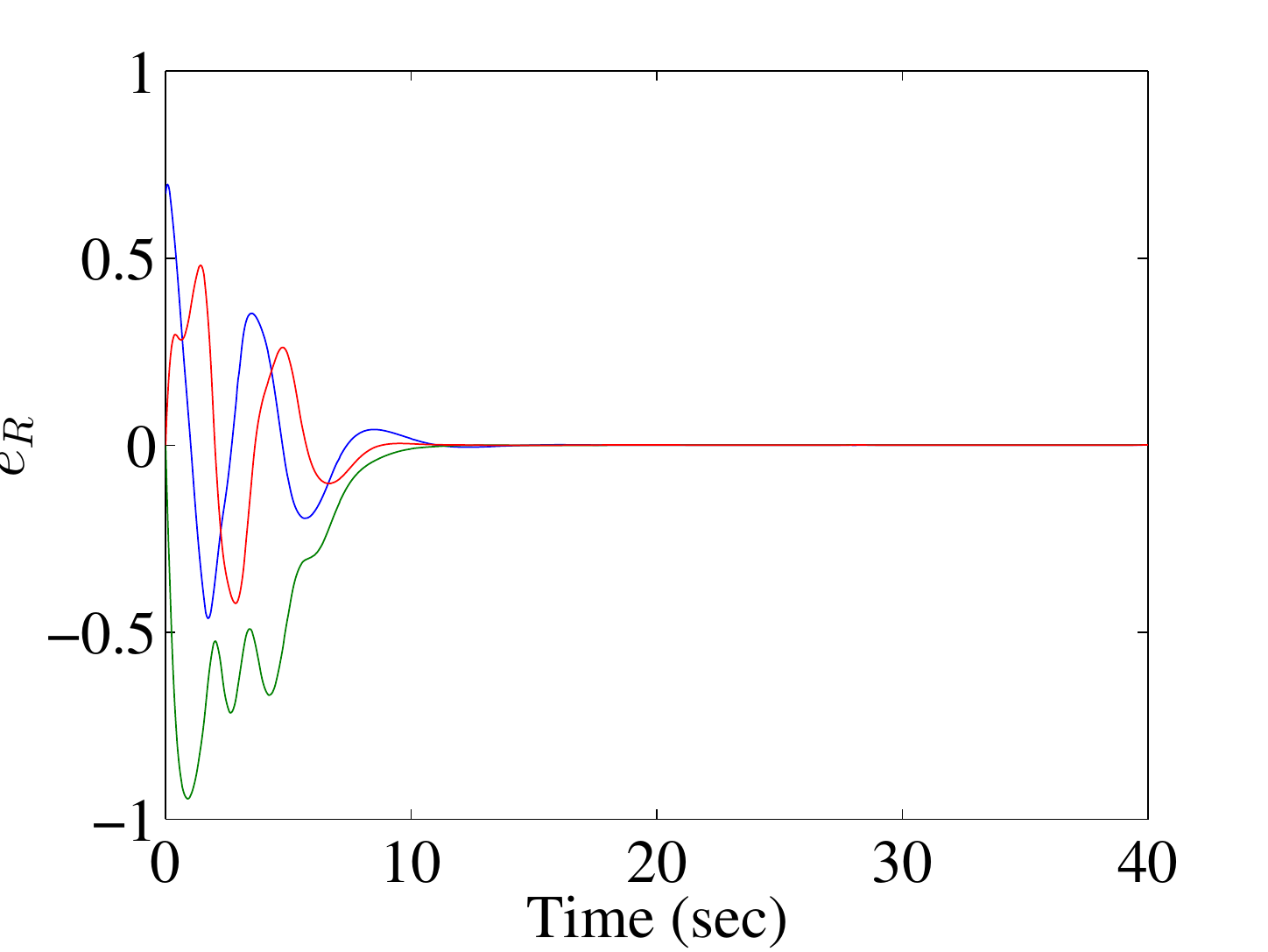}}
	\hspace*{-0.015\columnwidth}
	\subfigure[Angular velocity stabilization error $e_{\Omega}$]{\includegraphics[width=0.5\columnwidth]{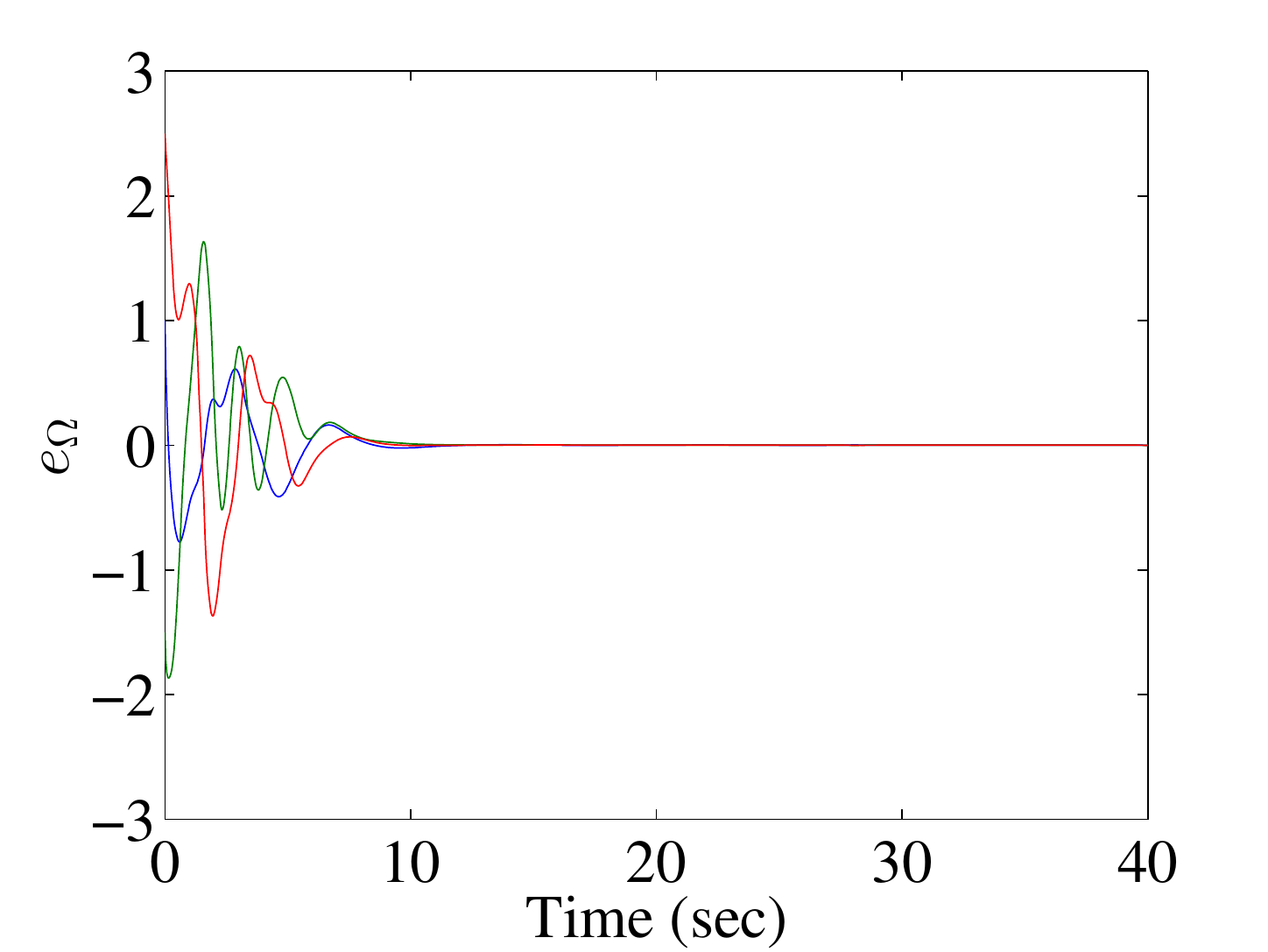}}
}
\centerline{
	\subfigure[Control moment $u$]{\includegraphics[width=0.5\columnwidth]{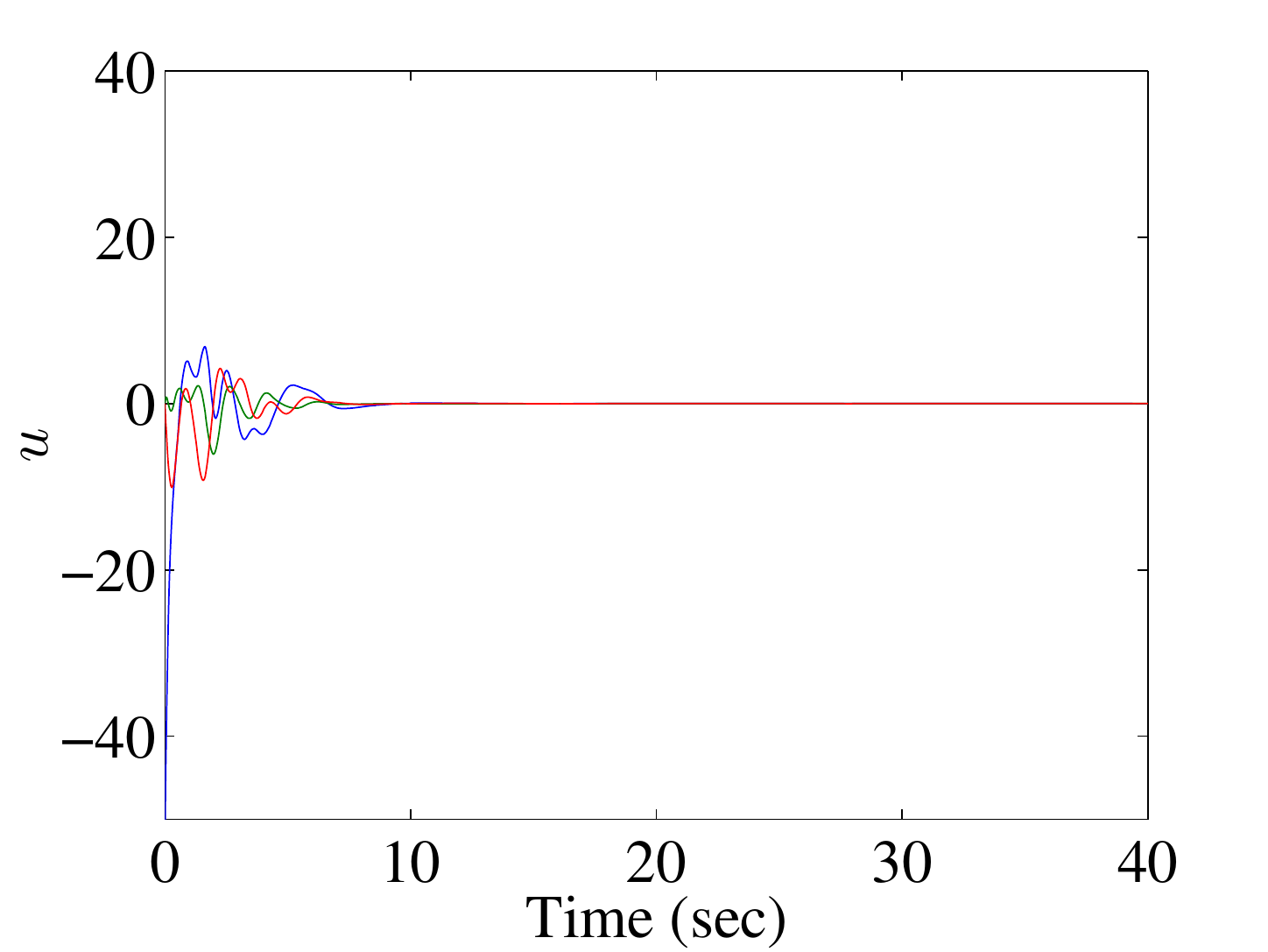}}
}
\caption{Attitude stabilization with the presented angular velocity observer on $\SO$}\label{fig:SO3-1}
\end{figure}

\begin{figure}[h]
\centerline{
	\hspace*{0.015\columnwidth}
	\subfigure[Attitude estimation error (scalar part of quaternion)]{\includegraphics[width=0.5\columnwidth]{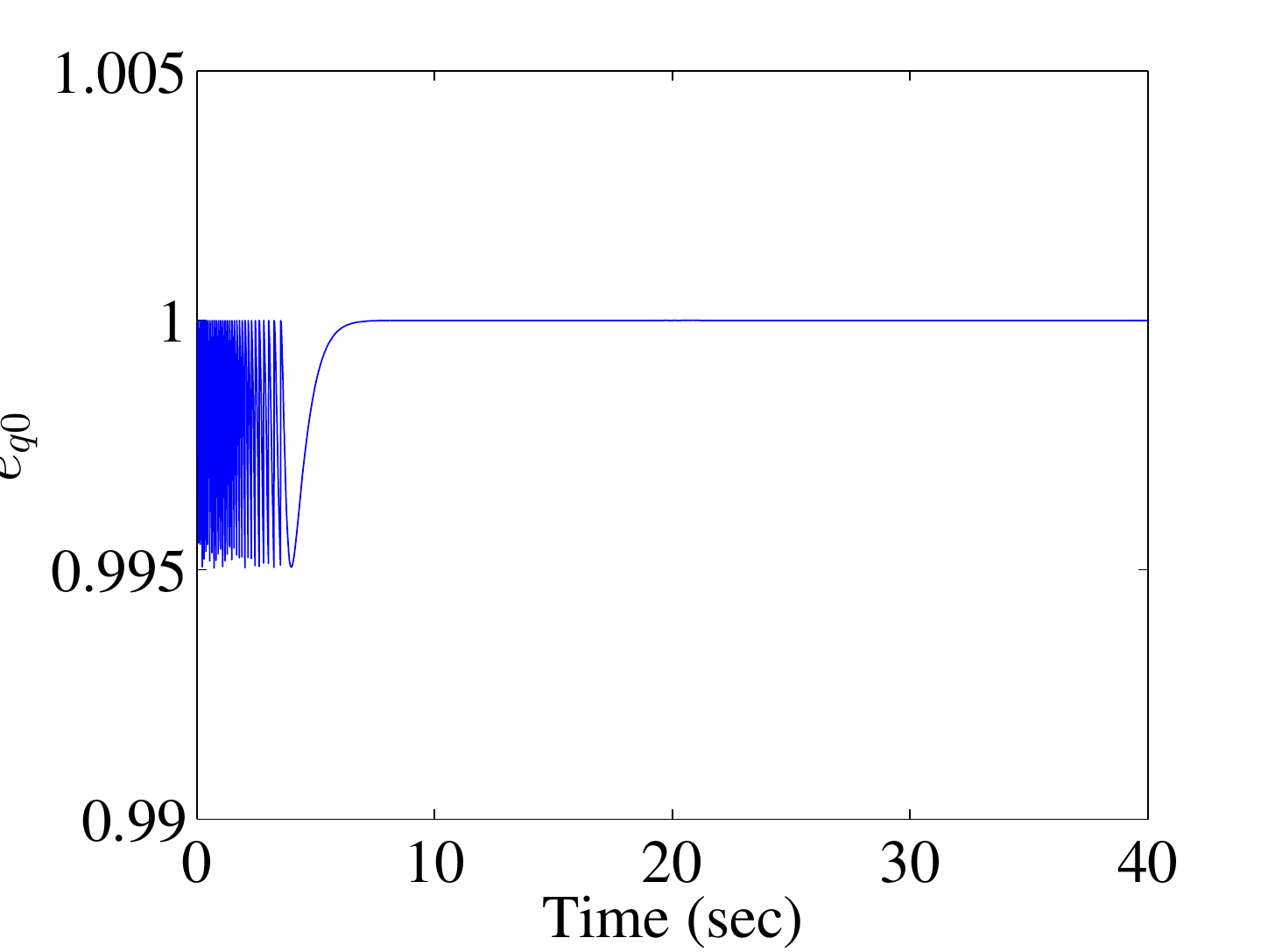}}
	\hspace*{-0.015\columnwidth}
	\subfigure[Attitude estimation error (vector part of quaternion)]{\includegraphics[width=0.5\columnwidth]{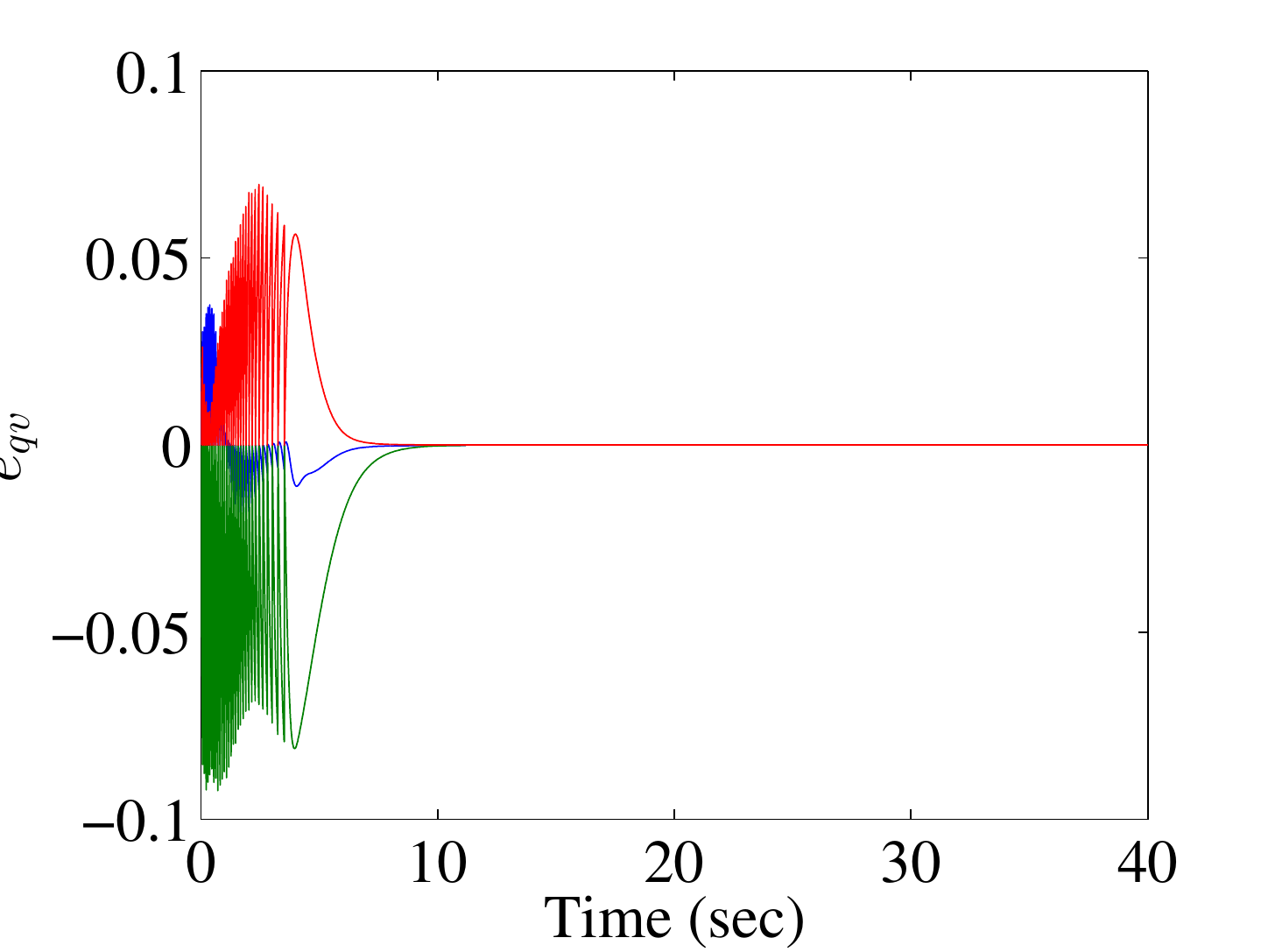}}
}	
\vspace*{-0.02\columnwidth}
\centerline{
	\hspace*{0.015\columnwidth}
	\subfigure[Angular velocity estimation error $\W-\bar{\W}$]{\includegraphics[width=0.5\columnwidth]{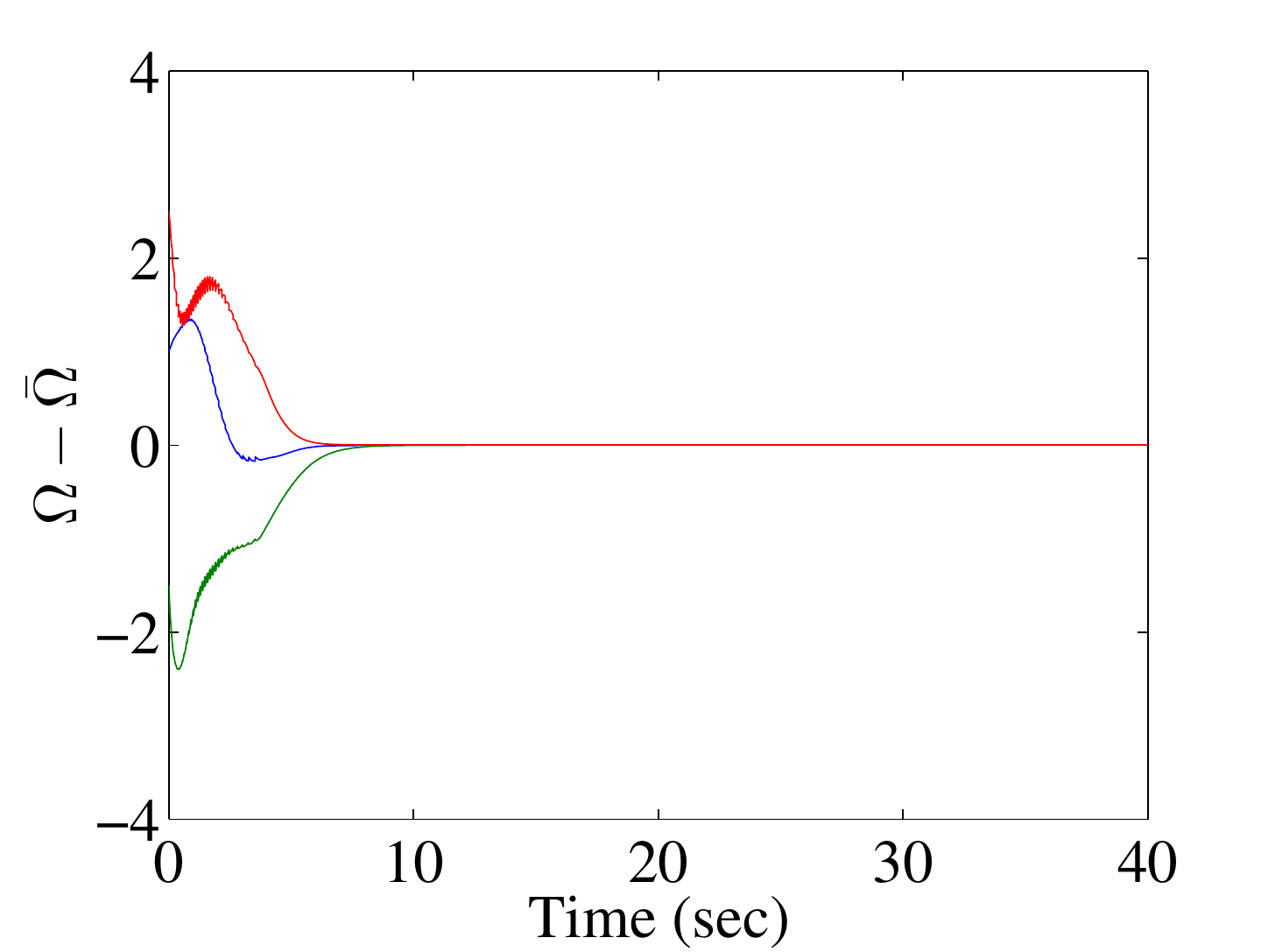}}
	\hspace*{-0.015\columnwidth}
	\subfigure[Control moment $u$]{\includegraphics[width=0.5\columnwidth]{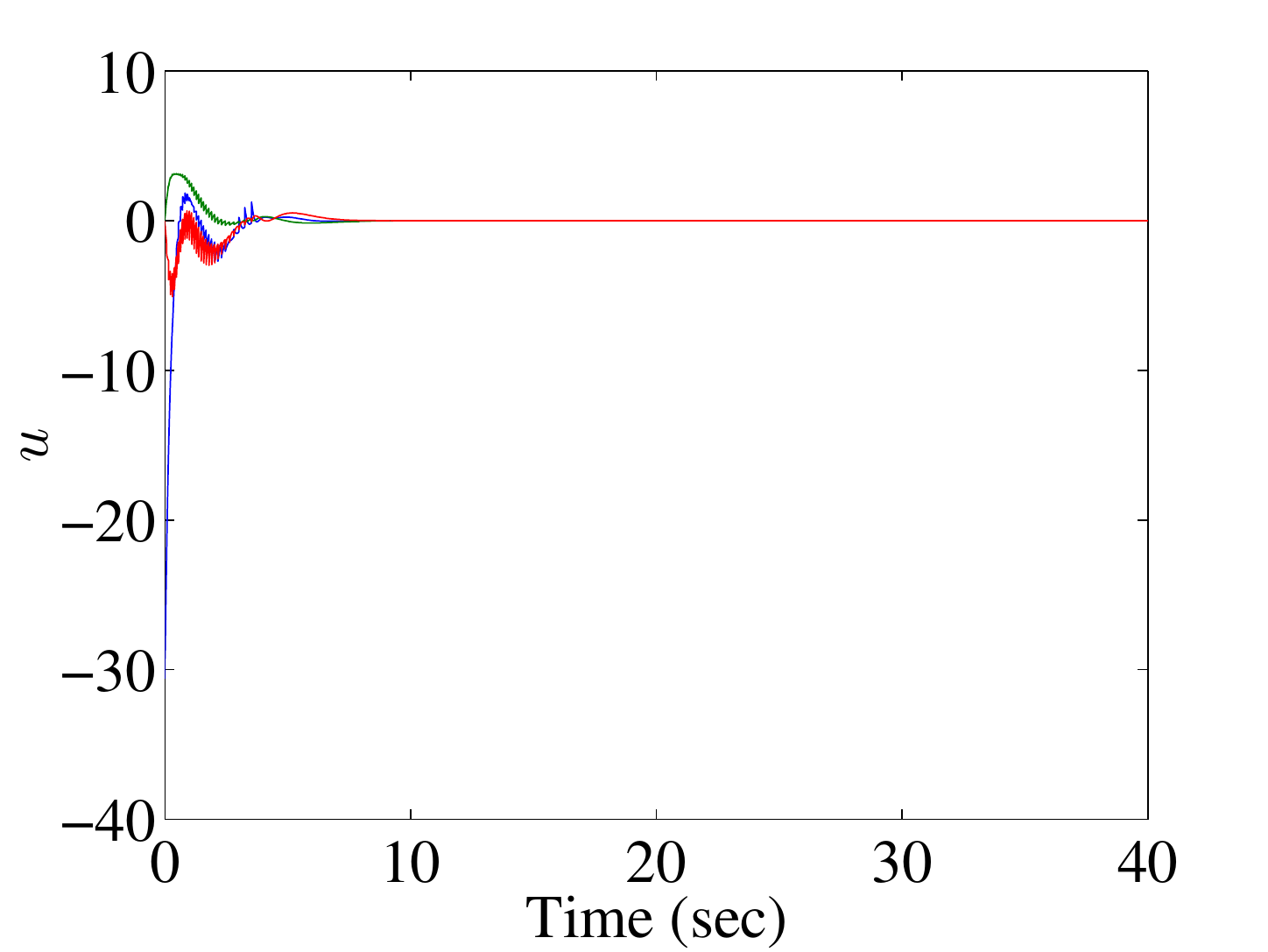}}
}
\vspace*{-0.02\columnwidth}
\centerline{
	\hspace*{0.015\columnwidth}
	\subfigure[Attitude stabilization error (scalar part of quaternion)]{\includegraphics[width=0.5\columnwidth]{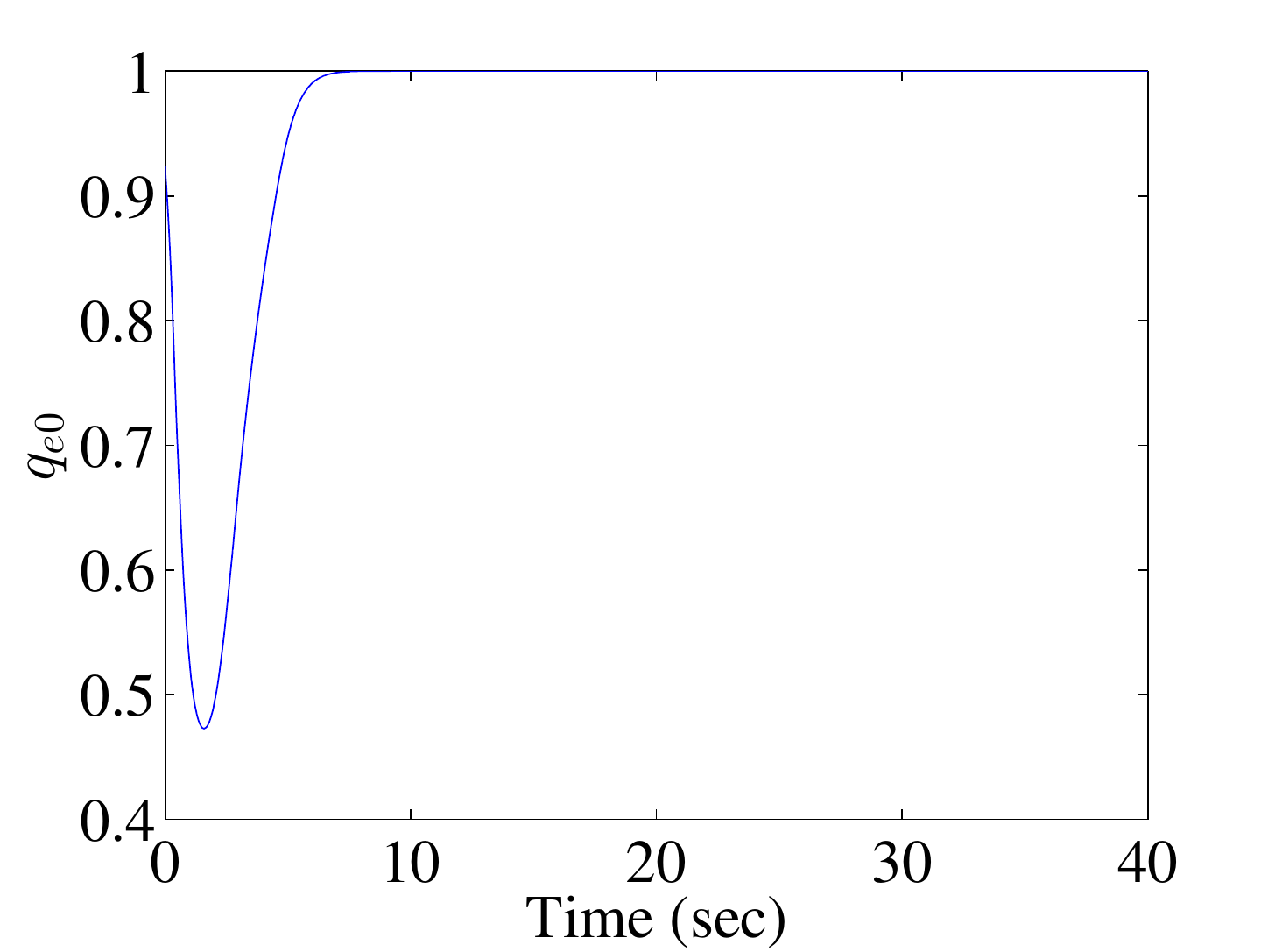}}
	\hspace*{-0.015\columnwidth}
	\subfigure[Attitude stabilization error (vector part of quaternion)]{\includegraphics[width=0.5\columnwidth]{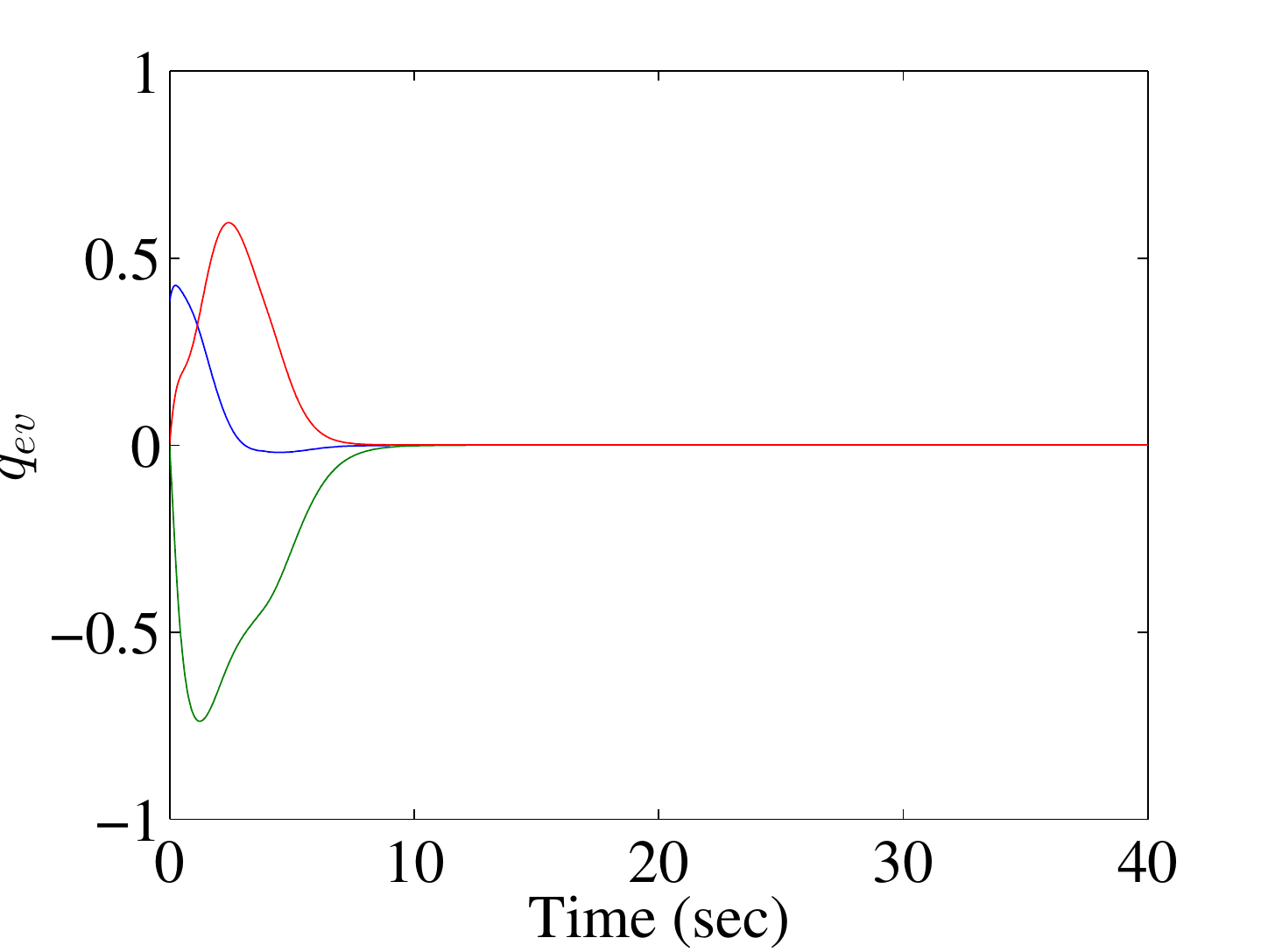}}
}
\caption{Attitude stabilization with the switching angular velocity observer in~\cite{Chu2014}}\label{fig:switch-1}
\end{figure}

\subsection{Attitude Tracking}
Next, we consider an attitude tracking problem. The desired attitude is given in terms of 3-2-1 Euler angles as $R_d(t)=R_d(\alpha(t), \beta(t), \gamma(t))$ where $\alpha(t)$, $\beta(t)$, and $\gamma(t)$ are chosen as $1$, $\sin(0.05t)$ and $\cos(0.1t)+2$, respectively. The initial conditions and control gains are identical to the attitude stabilization example. The corresponding results are illustrated at Fig.~\ref{fig:SO3-2}, where both the estimation errors and the tracking errors converge to zero nicely.

However, when the switching angular velocity~\cite{Chu2014} is applied to the same tracking problem, there are persistent switchings over the entire simulation period as illustrated by Fig.~\ref{fig:switch-2}, and the estimation errors and tracking errors do not converge for the given simulation period of 40 seconds. Such frequent switchings or high-frequency oscillations in the control moment may excite unmodelled dynamics or increase sensitivity to noise. These comparisons illustrate the desirable numerical properties of the proposed angular velocity observer explicitly.

\begin{figure}[h]
\centerline{
	\hspace*{0.015\columnwidth}
	\subfigure[Attitude estimation error $e_{R_E}$]{\includegraphics[width=0.5\columnwidth]{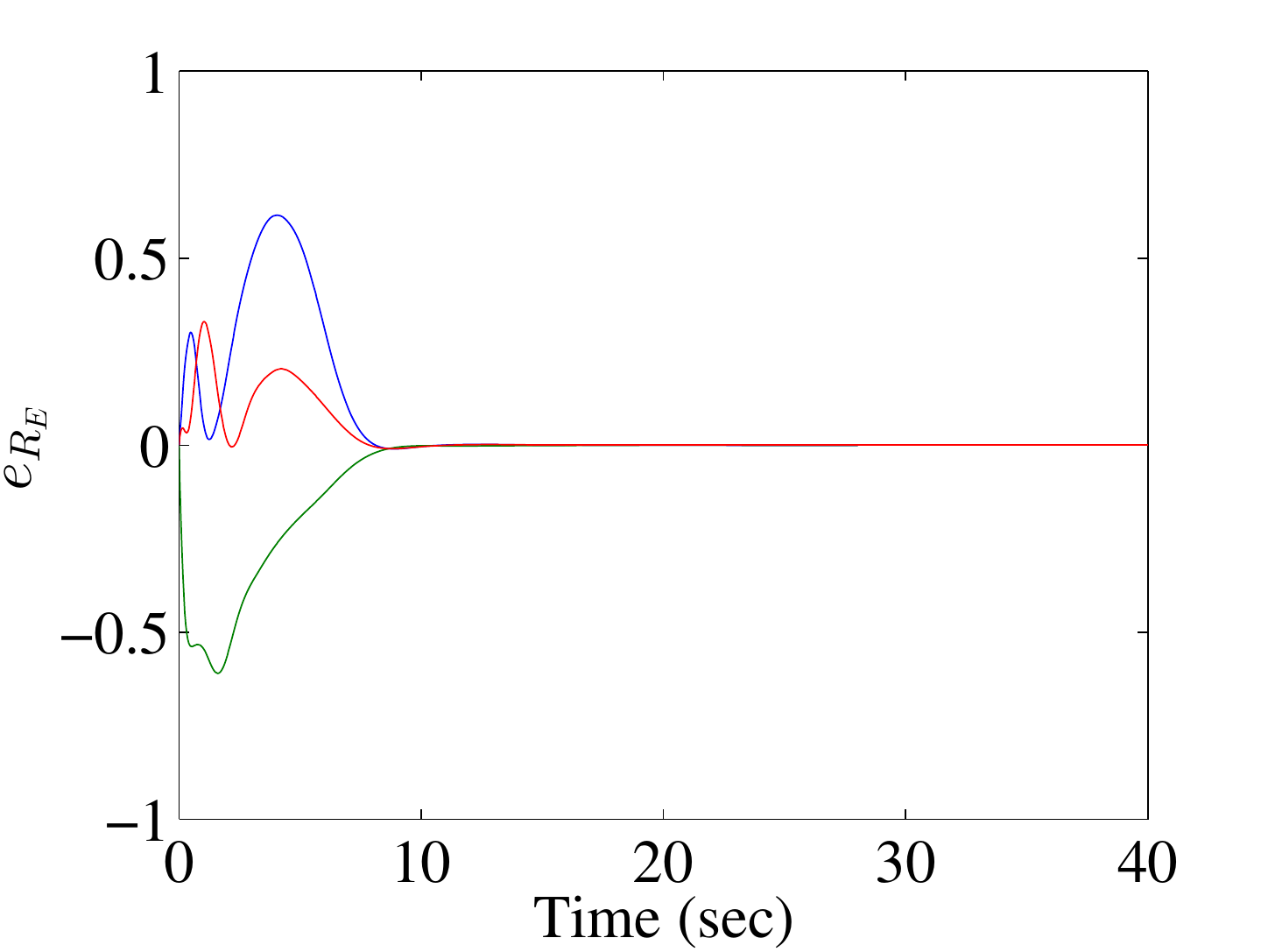}}
	\hspace*{-0.015\columnwidth}
	\subfigure[Angular velocity estimation error $\w-\bar{\w}$]{\includegraphics[width=0.5\columnwidth]{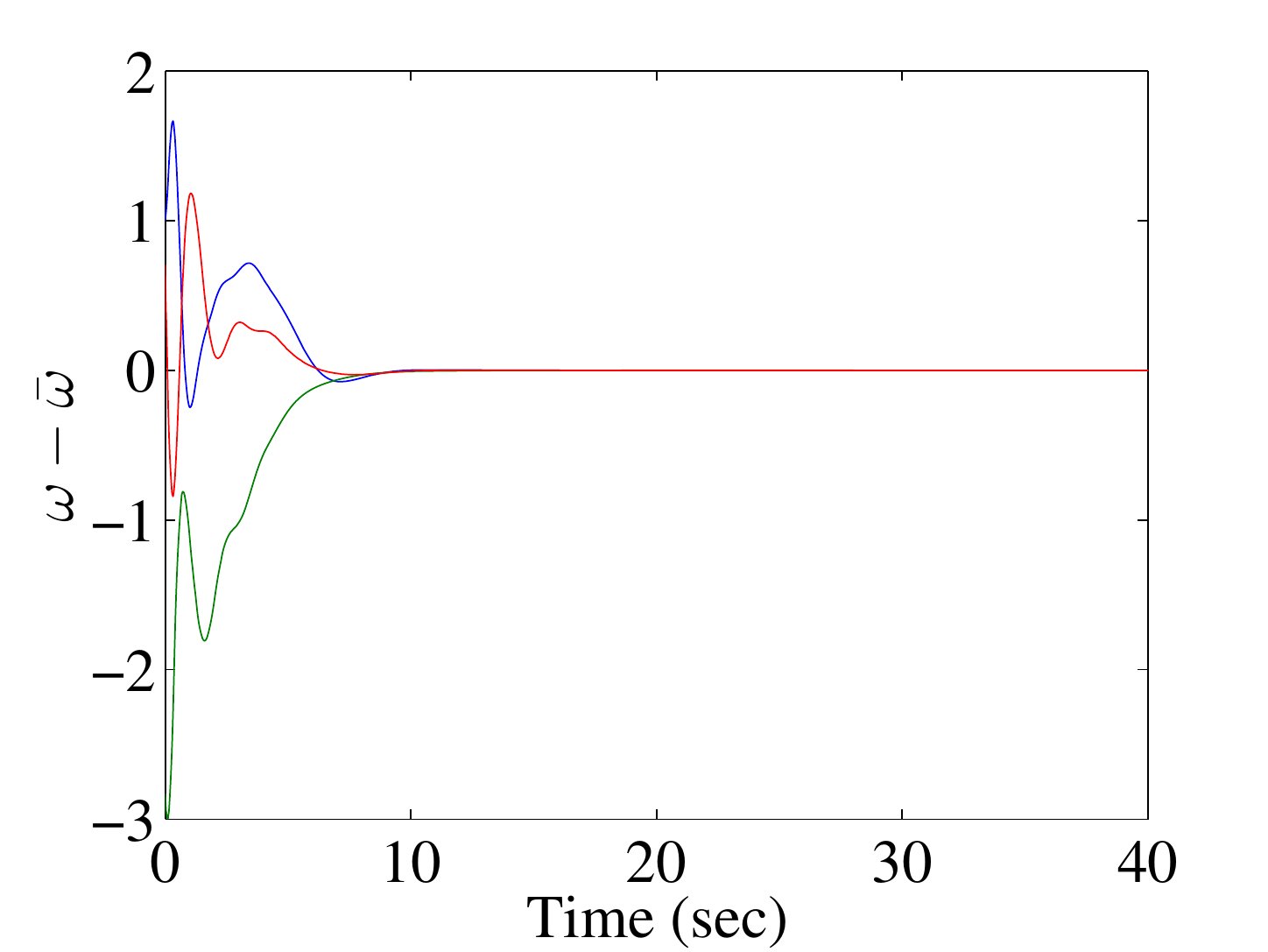}}
}	
\vspace*{-0.02\columnwidth}
\centerline{
	\hspace*{0.015\columnwidth}
	\subfigure[Attitude tracking error $e_R$]{\includegraphics[width=0.5\columnwidth]{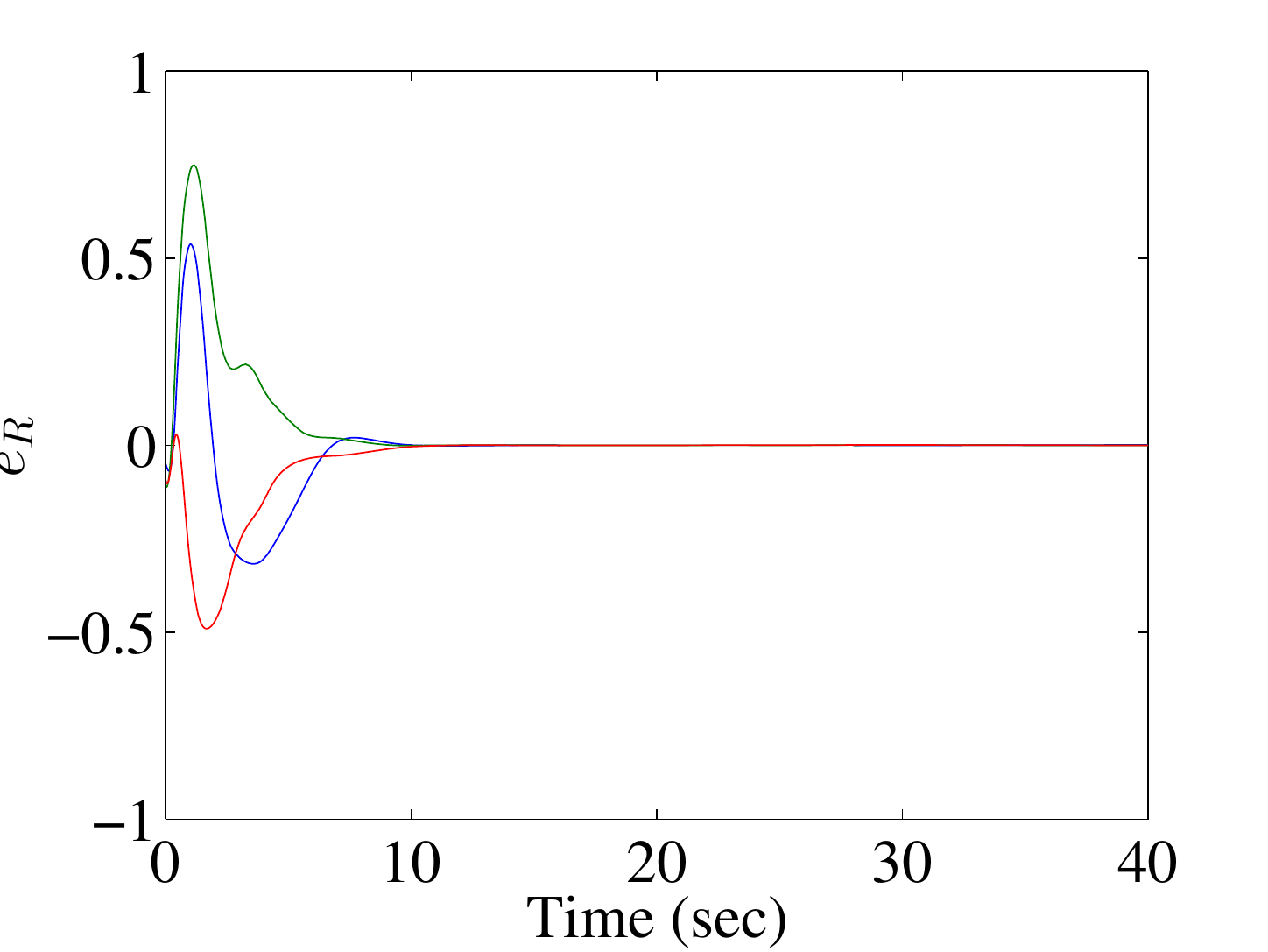}}
	\hspace*{-0.015\columnwidth}
	\subfigure[Angular velocity tracking error $e_{\Omega}$]{\includegraphics[width=0.5\columnwidth]{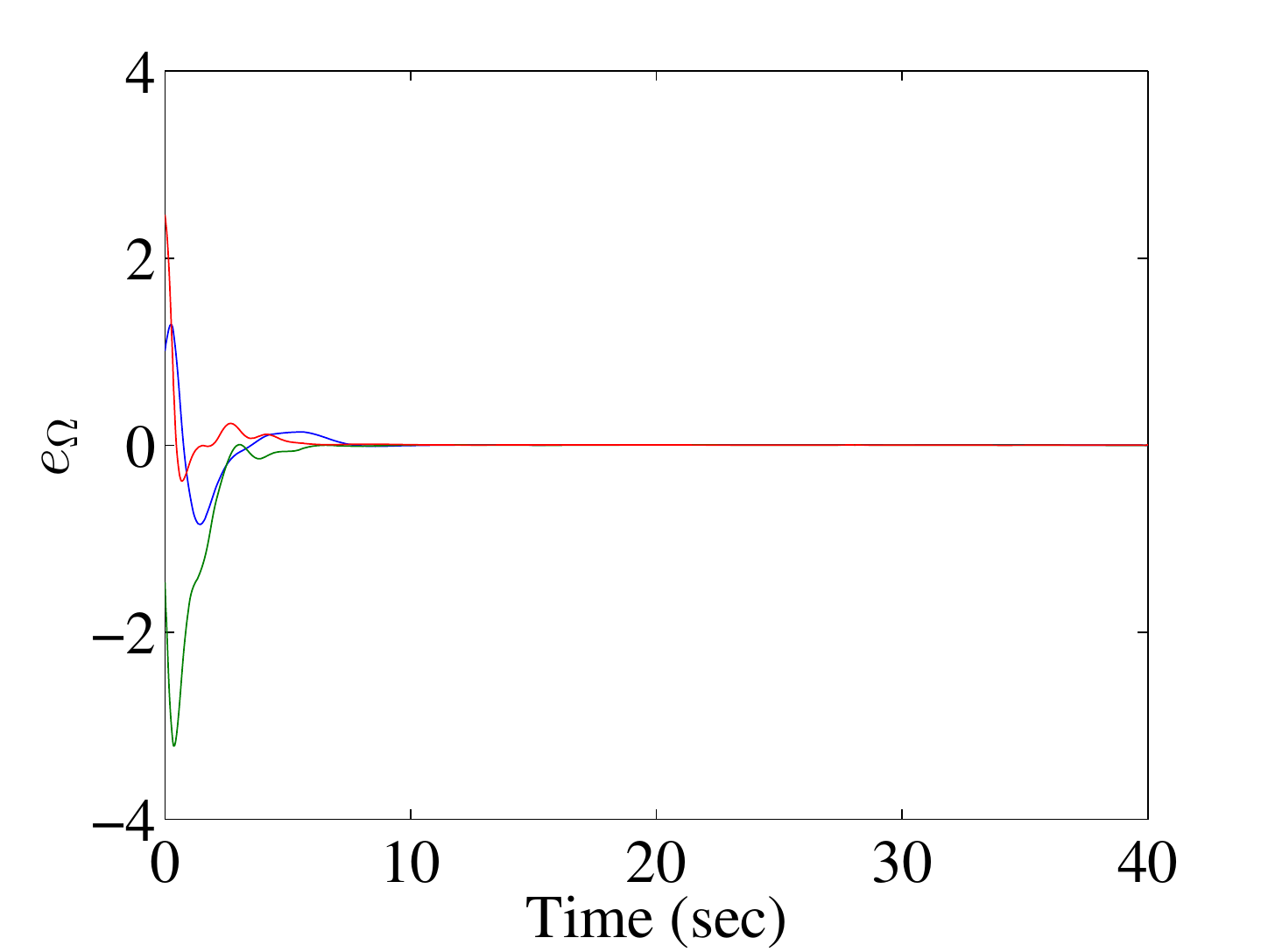}}
}
\centerline{
	\subfigure[Control moment $u$]{\includegraphics[width=0.5\columnwidth]{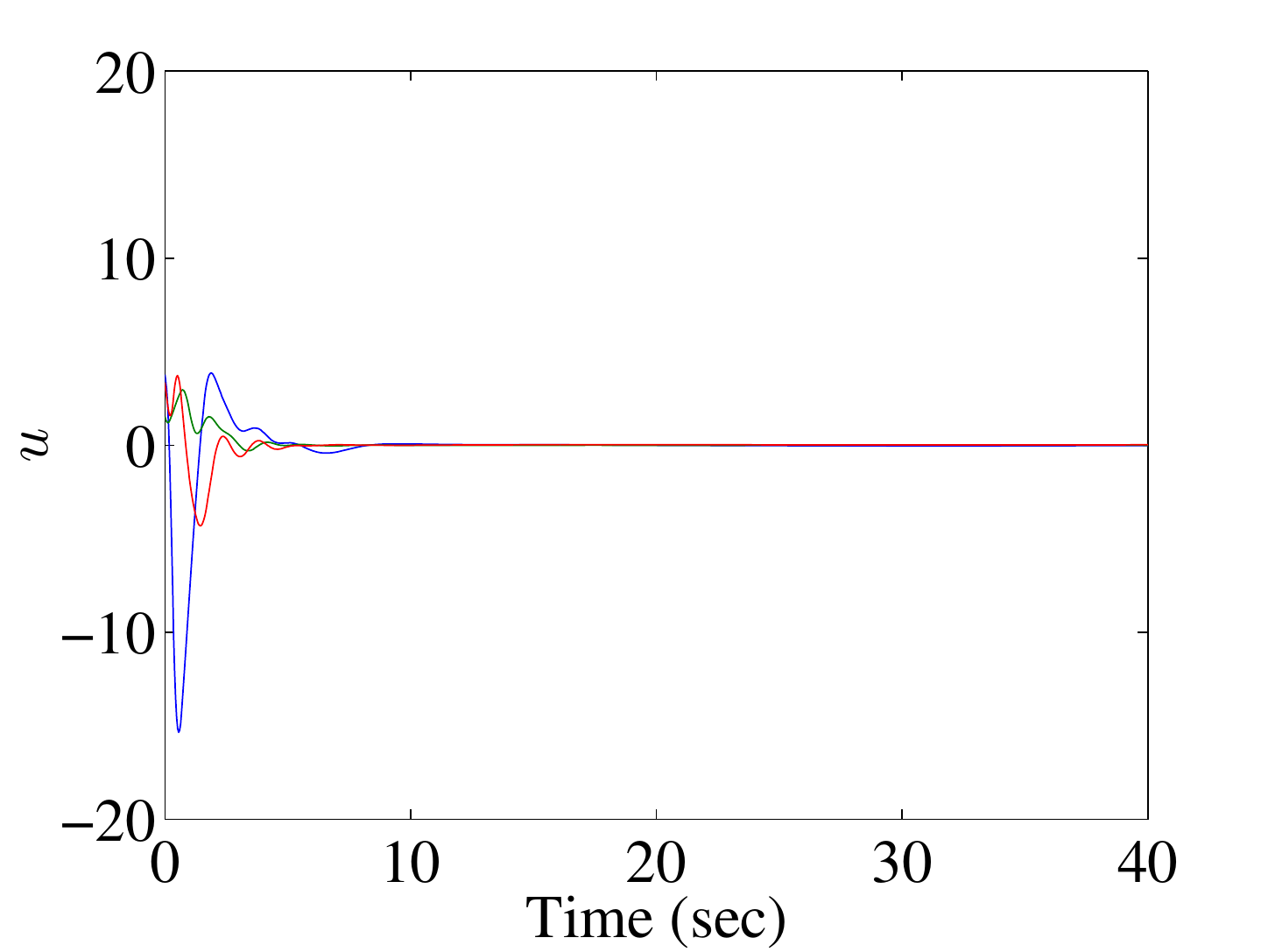}}
}
\caption{Attitude tracking with the presented angular velocity observer on $\SO$}\label{fig:SO3-2}
\end{figure}

\begin{figure}[h]
\centerline{
	\hspace*{0.015\columnwidth}
	\subfigure[Attitude estimation error (scalar part of quaternion)]{\includegraphics[width=0.5\columnwidth]{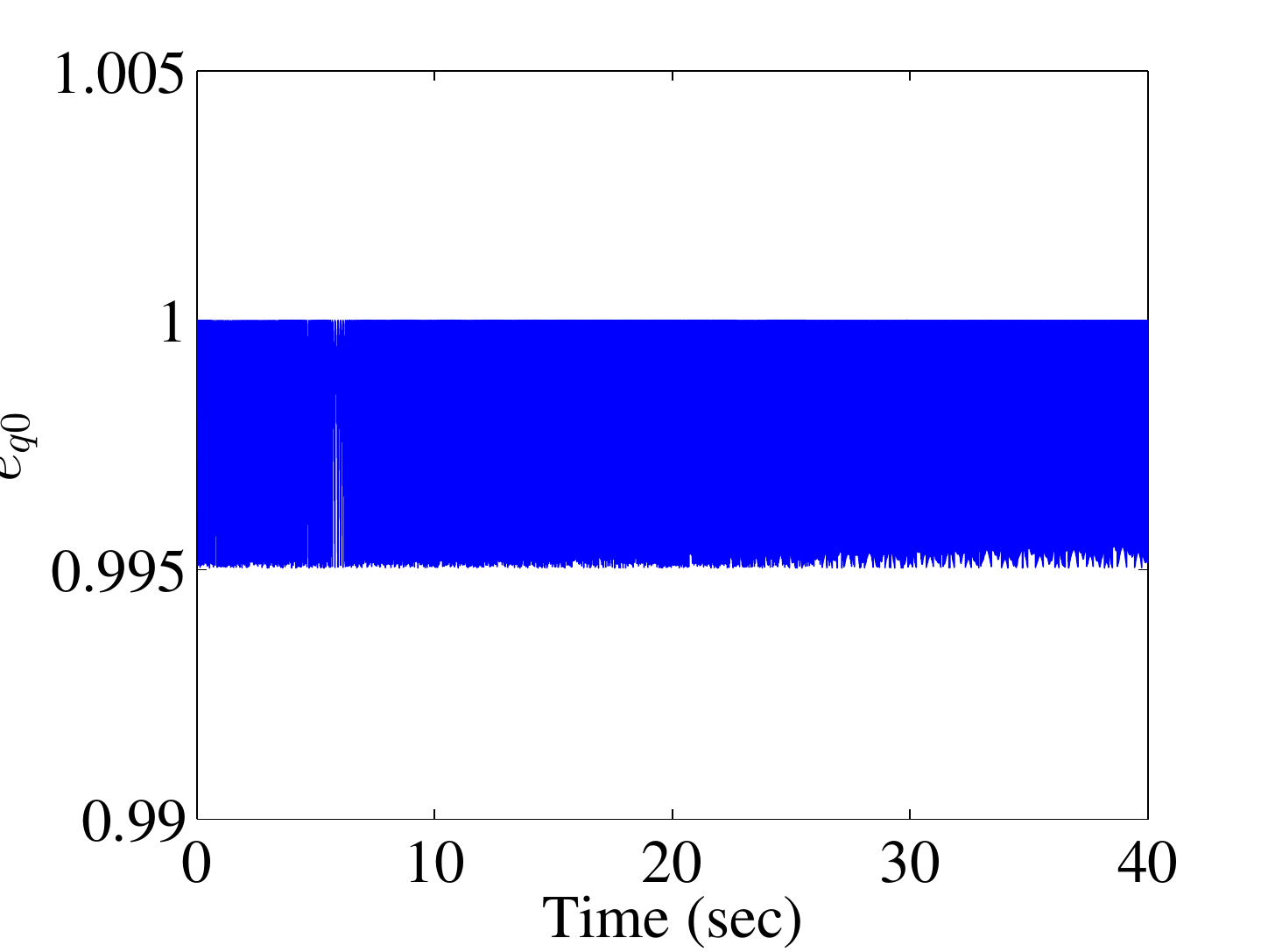}}
	\hspace*{-0.015\columnwidth}
	\subfigure[Attitude estimation error (vector part of quaternion)]{\includegraphics[width=0.5\columnwidth]{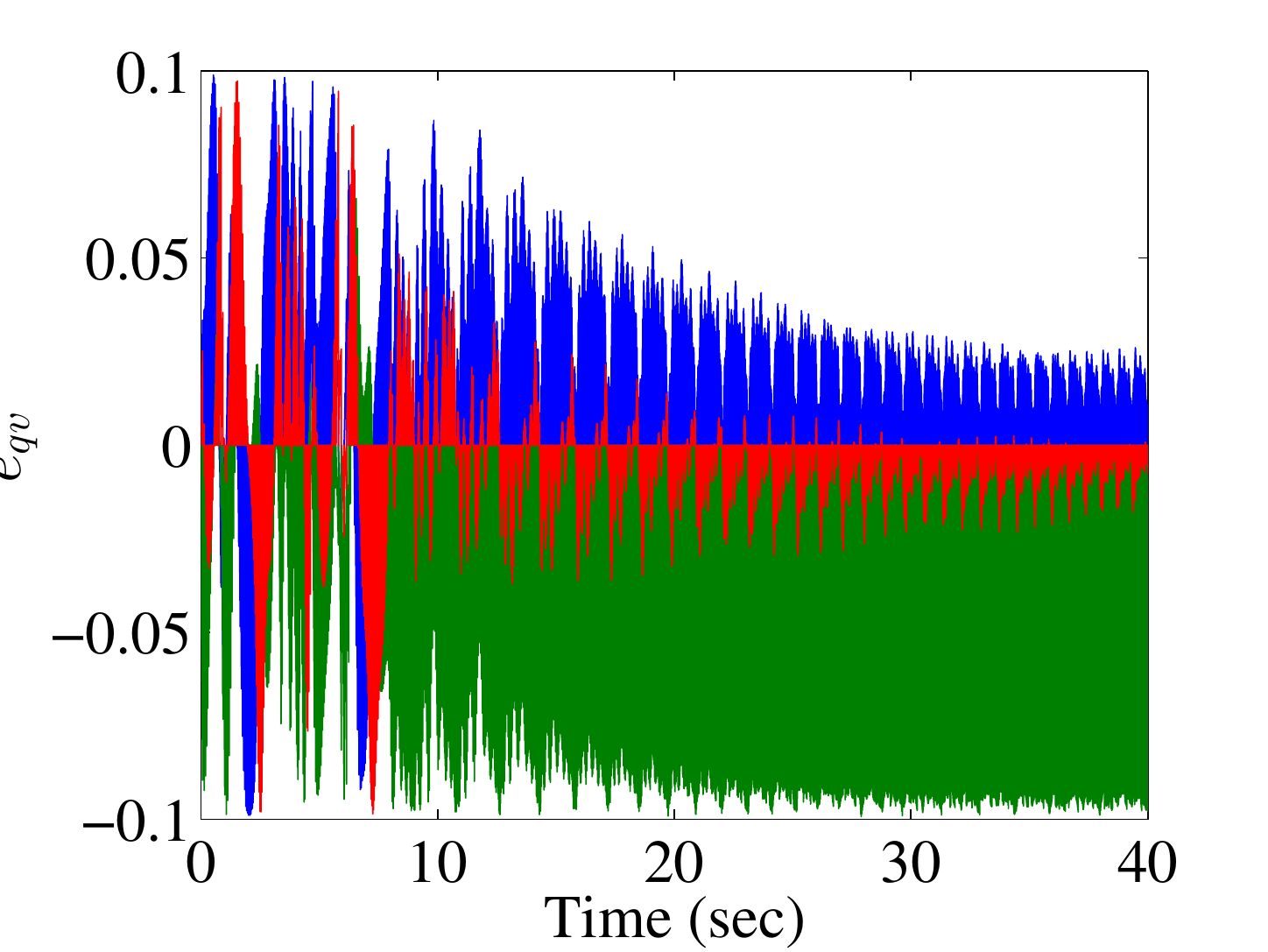}}
}	
\vspace*{-0.02\columnwidth}
\centerline{
	\hspace*{0.015\columnwidth}
	\subfigure[Angular velocity estimation error $\W-\bar{\W}$]{\includegraphics[width=0.5\columnwidth]{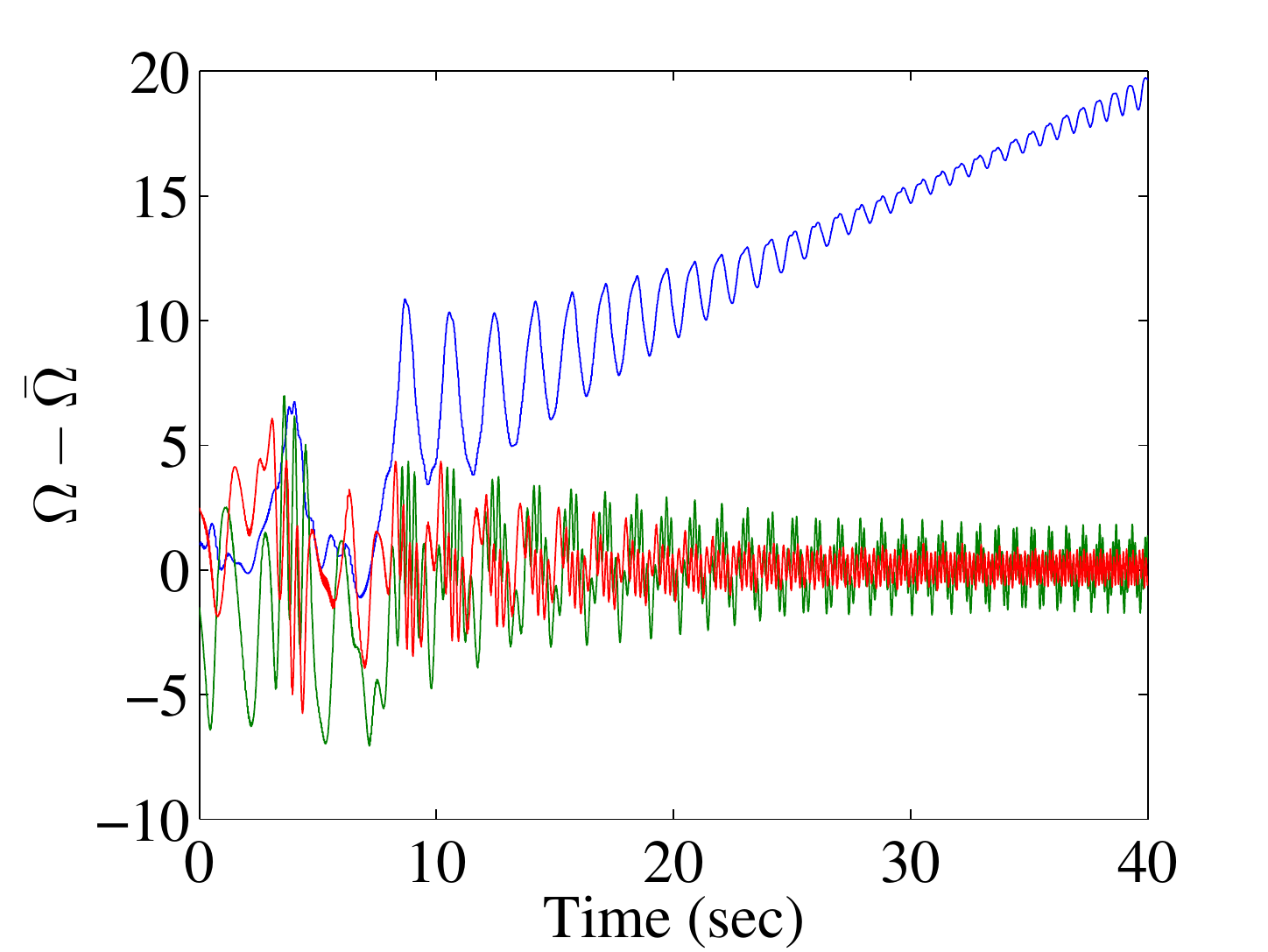}}
	\hspace*{-0.015\columnwidth}
	\subfigure[Control moment $u$]{\includegraphics[width=0.5\columnwidth]{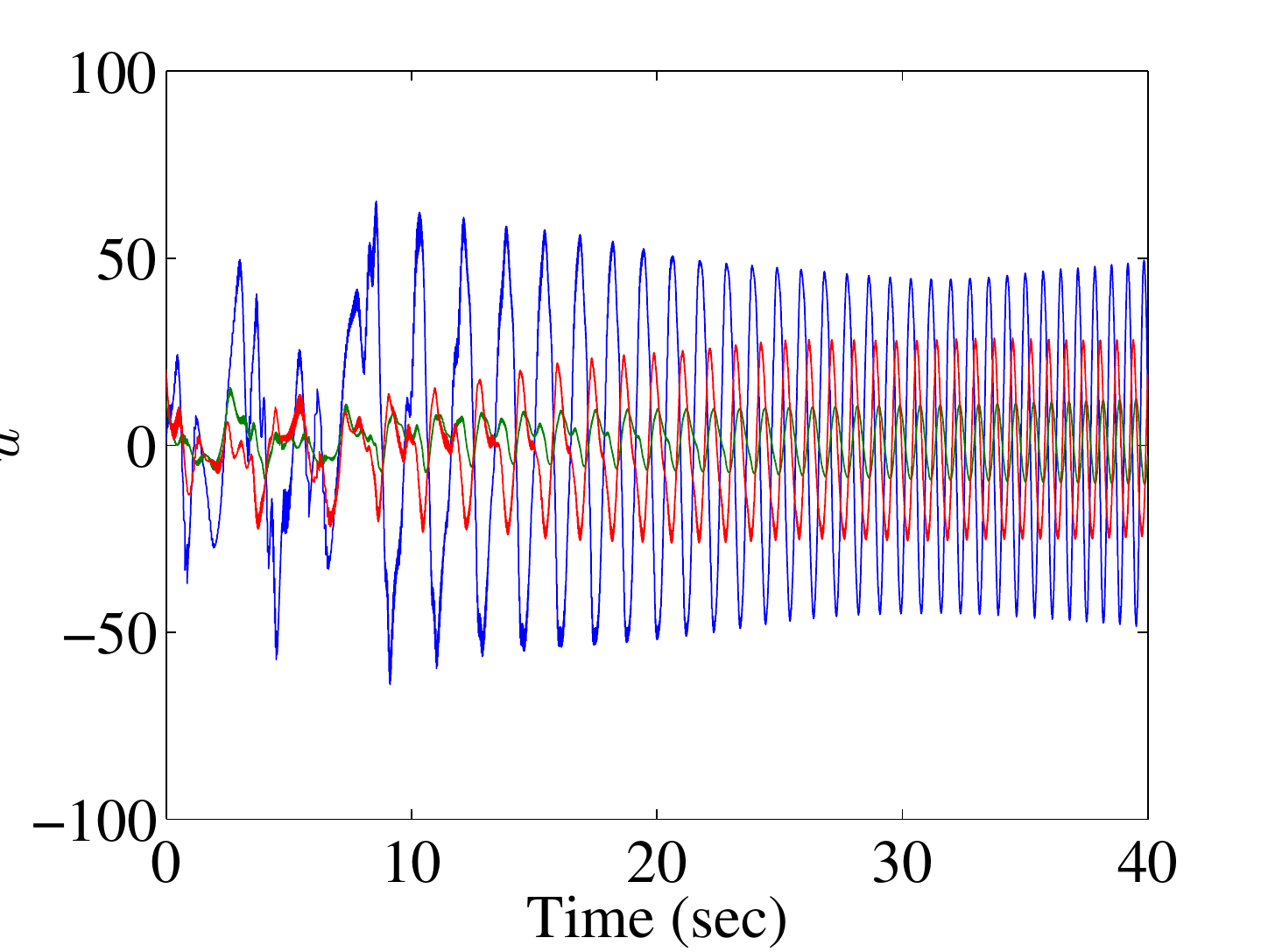}}
}
\vspace*{-0.02\columnwidth}
\centerline{
	\hspace*{0.015\columnwidth}
	\subfigure[Attitude tracking error (scalar part of quaternion)]{\includegraphics[width=0.5\columnwidth]{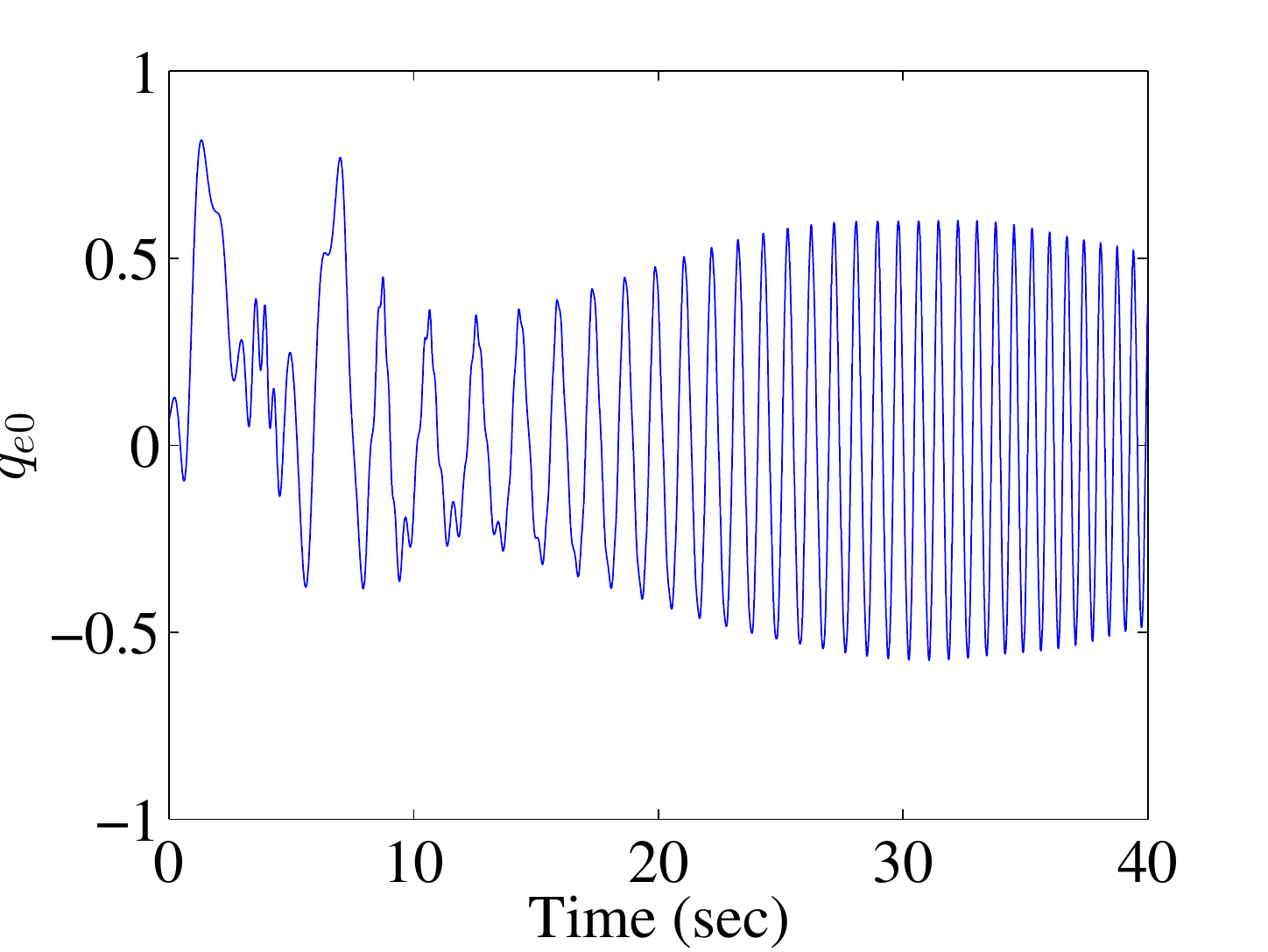}}
	\hspace*{-0.015\columnwidth}
	\subfigure[Attitude tracking error (vector part of quaternion)]{\includegraphics[width=0.5\columnwidth]{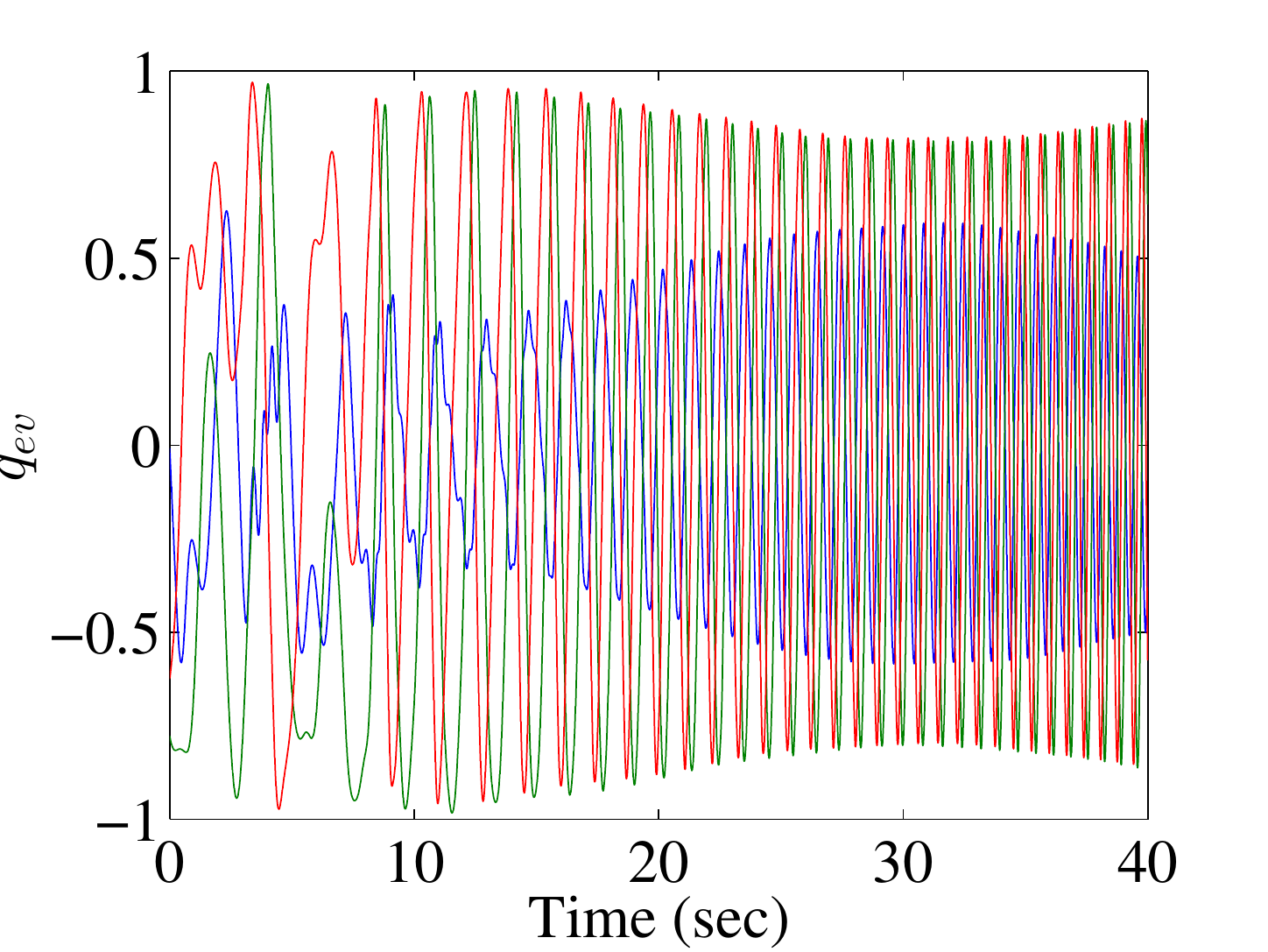}}
}	
\caption{Attitude tracking with the switching angular velocity observer in~\cite{Chu2014}}\label{fig:switch-2}
\end{figure}

\appendix[Proof of Proposition~\ref{prop:sep}]

For the weighting matrix $G=\mathrm{diag}[g_1,g_2,g_2]$ of the control system given at \refeqn{Psi}, let $\psi$ be a positive constant satisfying $\psi < \min\{g_2+g_3,g_3+g_1,g_1+g_2\}$.  Consider the following domain for the configuration of the attitude dynamics and the observer:
\begin{align*}
D=\{(R,\Omega,\bar R,\bar\Omega)\in (\SO\times\Re^3)^2\,|\, 
\Psi < \psi, \Psi_E <\bar\psi_E\},
\end{align*}
The subsequent stability proof is developed in this domain. We first show that the estimated attitude and angular velocity trajectories starting from the initial conditions satisfying \refeqn{roa1} and \refeqn{roa2} satisfy $\Psi_E < \bar\psi_E$ always, i.e., the estimated tragictory stay in the domain $D$. Recall the Lyapunov function $\mathcal{U}$ given at \refeqn{U}. For the initial estimates $\bar R(0)$ and $\bar\omega(0)$ satisfying \refeqn{roa1} and \refeqn{roa2}, we have
\begin{align*}
\mathcal{U}(0) < k_E (\bar\psi_E - \Psi_E(0)) + \Psi_E(0) = k_E\bar\psi_E.
\end{align*}
As $\mathcal{U}(t)$ is non-increasing from \refeqn{Udot}, we have
\begin{align*}
k_E \Psi_E(t) \leq \mathcal{U}(t) \leq \mathcal{U}(0) < k_E\bar\psi_E,
\end{align*}
which yields
\begin{align*}
\Psi_E(t) < \bar\psi_E < \min\{n_1,\frac{1}{2}(\tr[G_E]-\frac{\lam_M}{\lam_m}\|G_E\|)\}
\end{align*}
for all $t \geq 0$.

Consider the following Lyapunov function
	\begin{gather} 
	\V=\V_c+\V_o,
	\end{gather}
where $\V_c$ and $\V_o$ are related to the controller and the observer, respectively, and they are defined as
\begin{align*}
\V_c &=\frac{1}{2}e_\W^\T J_0e_\W +k_R\Psi +c_1e_R^\T J_0e_\W,\\
\V_o &=\U-c_2e_{\w_E}^\T e_{R_E}=\|e_{\w_E}\|^2+k_E\Psi_E-c_2e_{\w_E}^\T e_{R_E},
\end{align*}
for positive constants $c_1$ and $c_2$. It has been shown that $\V_c$ is positive definite about $(e_\W,e_R)=(0,0)$ when $c_1$ is sufficiently small, and it satisfies
\begin{align}
\alpha^T M_1 \alpha \leq \V_c \leq \alpha^T M_2 \alpha,\label{eq:VcB}
\end{align}
where $\alpha=[\|e_\Omega\|, \|e_R\|]^T\in\Re^2$, and the matrices $M_1,M_2\in\Re^{2\times 2}$ are defined as
\begin{align*}
M_1 = \frac{1}{2}\begin{bmatrix} \lambda_m & -c_1\lambda_M \\ 
-c_1\lambda_M & 2b_1 k_R\end{bmatrix},\;
M_2 = \frac{1}{2}\begin{bmatrix} \lambda_M & c_1\lambda_M \\ 
c_1\lambda_M & 2b_2 k_R\end{bmatrix},\quad
\end{align*}
for a constant $b_1,b_2$ that can be determined by the weighting matrix $G$ and $\psi$~\cite{LeeCST13}.

Similarly, from \refeqn{PsiE}, the second part of the Lyapunov function satisfies
	\begin{gather} 
	\xi^\T M_3\xi\leq\V_o\leq\xi^\T M_4\xi,
	\end{gather}
where $\xi=[\|e_{R_E}\|, \|e_{\w_E}\|]^\T\in\Re^2$ and 
	\begin{gather*} 
	M_3=\frac{1}{2}\begin{bmatrix} k_E\frac{2n_1}{n_2+n_3} & -c_2 \\ -c_2 & 2 \end{bmatrix},\;
	M_4=\frac{1}{2}\begin{bmatrix} k_E\frac{2n_1n_4}{n_5(n_1-\psi_E)} & c_2 \\ c_2 & 2 \end{bmatrix}.
	\end{gather*}
If the constant $c_2$ is chosen sufficiently small such that
	\begin{gather} 
	c_2<\mathrm{min}\{2\sqrt{\frac{k_E n_1}{n_2+n_3}},\, 2\sqrt{\frac{k_En_1n_4}{n_5(n_1-\psi_E)}}\}, \label{eq:c1}
	\end{gather}
then the matrices $M_3$ and $M_4$ are positive definite. From these, the Lyapunov function $\V$ is positive definite about $(e_\W,e_R,e_{\w_E},e_{R_E})=(0,0,0,0)$ and it is decrescent.

From \refeqn{ed1}-\refeqn{ed3} and \refeqn{free}, the time-derivative of $\V_c$ is
	\begin{align} 
	\dot\V_c
	&= (e_\W+c_1e_R)^\T J_0\dot{e}_\W +k_R\dot{\Psi} +c_1\dot{e}_R^\T J_0e_\W \no\\
%	&= (e_\W+c_1e_R)^\T(-k_Re_R -k_\W\bar{e}_\W +\xi\t e_\W) \no\\
	&\quad +k_Re_R^\T e_\W +c_1(E_ce_\W)^\T J_0e_\W \no\\
	&= -k_\W(e_\W+c_1e_R)^\T\bar{e}_\W-c_1k_R\|e_R\|^2 \no\\
	&\quad +c_1e_R^\T\hat{\chi}e_\W  +c_1(E_ce_\W)^\T J_0e_\W. \label{eq:Vc_dot0}
	\end{align}
The following properties of the error variables has been shown in~\cite{IEPC2013}:
	\begin{gather*} 
	\|E_c\|\leq \frac{1}{\sqrt{2}}\tr[G], \\
	\|\chi\|\leq \lam_M\|e_\W\|+B_1^*, \\
	\|e_R\|\leq B_2^*\triangleq \frac{1}{2}\sqrt{12p_2+3p_3},
	\end{gather*}
where the constants $p_2,p_3$ are defined as
		\begin{align*} 
		p_2 &= \mathrm{max}\{(g_1-g_2)^2,\, (g_2-g_3)^2,\, (g_3-g_1)^2\}, \\
		p_3 &= \mathrm{max}\{(g_1+g_2)^2,\, (g_2+g_3)^2,\, (g_3+g_1)^2\}. 
		\end{align*}
Applying these bounds to \refeqn{Vc_dot0}, we obtain
	\begin{align}
	\dot\V_c
	&\leq -k_\W(e_\W+c_1e_R)^\T\bar{e}_\W-c_1k_R\|e_R\|^2 \no\\ 
	&\quad +\frac{1}{2}c_1\lam_M(\sqrt{2}\tr[G]+B_2^*)\|e_\W\|^2 \no\\
	&\quad  +c_1(k_\W+B_1^*)\|e_R\|\|e_\W\|. \label{eq:dVc1} 
	\end{align}
From \refeqn{ev3} and the second part of \refeqn{eom1}, we can write $e_{\w_E}=RJ_0R^\T(\w-\bar\w)$. Therefore, we have
	\begin{align*} 
	R^\T e_{\w_E}=J_0R^\T(\w-\bar\w)=J_0(\W-\bar{\W}), 
	\end{align*}
which follows that
	\begin{align*} 
	\bar{\W}=\W-J_0^{-1}R^\T e_{\w_E}.
	\end{align*}
From this, the angular velocity error vector given by \refeqn{eWbar} can be rewritten as
	\begin{align} 
	\bar{e}_\W=e_\W-J_0^{-1}R^\T e_{\w_E}. \label{eq:bW}
	\end{align}
Substituting \refeqn{bW} into \refeqn{dVc1}, we obtain	
	\begin{align} 
	\dot\V_c 
	&\leq -B_3^*\|e_\W\|^2  -c_1k_R\|e_R\|^2 +c_1(k_\W+B_1^*)\|e_R\|\|e_\W\|	\no\\
	&\quad +k_\W\frac{1}{\lam_m}\|e_\W\|\|\mu\| +c_1k_\W \frac{1}{\lam_m}\|e_R\|\|\mu\|
	\label{eq:Vcdot}
	\end{align}
where $B_3^*=[k_\W-\frac{c_1\lam_M}{2}(\sqrt{2}\,\tr[G]+B_2^*)]$.	

Next, we find the time-derivative of $\V_o$. From \refeqn{Udot} and properties (v), (vi) of Proposition \ref{prop:erdyn}, we have
	\begin{align} 
	\dot{\V}_o
	&=\dot{\U}-c_2\dot{e}_{\w_E}^\T e_{R_E} -c_2e_{\w_E}^\T\dot{e}_{R_E} \no\\
	&= \dot{\U}+\frac{1}{2}c_2k_Ee_{R_E}^\T J^{-1}e_{R_E} \no\\ 
	&\quad -\frac{1}{2}c_2e_{\w_E}^\T(\tr[Q_EG_E]\I-Q_EG_E)\w_E.	\label{eq:dVo}
	\end{align}
Equation \refeqn{wE} can be rewritten as
	\begin{align}
	\w_E
	&=J^{-1}[J(\w-\bar\w)-k_ve_{R_E}] \no\\
	&=J^{-1}(e_{\w_E}-k_ve_{R_E}).\label{eq:wE2} 
	\end{align}
Substituting \refeqn{wE2} and \refeqn{Udot} into \refeqn{dVo}, we obtain
	\begin{align} 
	\dot{\V}_o
	&\leq -(\frac{k_v}{\lam_M}-\frac{c_2}{2\lam_m})k_E\|e_{R_E}\|^2 \no\\
	&\quad -(\frac{\tr[Q_EG_E]}{\lam_M}-\frac{\|G_E\|}{\lam_m})\frac{c_2}{2}\|e_{\w_E}^2\| \no\\
	&\quad +\frac{\tr[Q_EG_E]+\|G_E\|}{2\lam_m}c_2k_v\|e_{\w_E}\|\|e_{R_E}\|. \label{eq:Vodot}
%&= -\frac{1}{2\lam_m}\big[k_E\frac{2k_v\lam_m-c\lam_M}{\lam_M}\|e_{R_E}\|^2 \\
%&\quad -ck_v(\tr[Q_E]+1)\|e_{\w_E}\|\|e_{R_E}\| \\
%&\quad +\frac{c\tr[Q_E]\lam_m-c\lam_M}{\lam_M}\|e_{\w_E}^2\|\big] 
	\end{align}	
Combining \refeqn{Vcdot} and \refeqn{Vodot}, the time-derivative of the Lyapunov function satisfies
	\begin{align*} 
	\dot{\V}
 	&\leq -\frac{B_3^*}{2}\|e_\W\|^2  -\frac{c_1}{2}k_R\|e_R\|^2 +c_1(k_\W+B_1^*)\|e_R\|\|e_\W\| \\
 	&\quad -\frac{c_1}{2}k_R\|e_R\|^2+\frac{c_1k_\W}{\lam_m}\|e_R\|\|e_{\w_E}\| -\frac{c_2A_E}{6\lam_M\lam_m}\|e_{\w_E}^2\|\\
	&\quad -\frac{B_3^*}{2}|e_\W\|^2 +\frac{k_\W}{\lam_m}\|e_\W\|\|e_{\w_E}\| -\frac{c_2A_E}{6\lam_M\lam_m}\|e_{\w_E}^2\|\\ 
	&\quad -\frac{2k_v\lam_m-c_2\lam_M}{2\lam_M\lam_m}k_E\|e_{R_E}\|^2 -\frac{c_2A_E}{6\lam_M\lam_m}\|e_{\w_E}^2\| \\
	&\quad +\frac{c_2k_v}{2\lam_m}B_E\|e_{\w_E}\|\|e_{R_E}\|,
	\end{align*}
where $A_E=\tr[Q_EG_E]\lam_m-\|G_E\|\lam_M\in\Re$ and $B_E=\tr[Q_EG_E]+\|G_E\|\in\Re$. Note that \refeqn{ratio} ensures that $A_E>0$. This is rearranged as the following matrix form:	
	\begin{align*} 
	\dot{\V}\leq -\al^\T W_1\al-\be^\T W_2\be -\zeta^\T W_3\zeta-\xi^\T W_4\xi, 
	\end{align*}
where $\al=[\|e_{\W}\|,\|e_{R}\|]^\T$, $\be=[\|e_R\|,\|e_{\w_E}\|]^\T$,  $\zeta=[\|e_{\W}\|,\|e_{\w_E}\|]^\T$, $\xi=[\|e_{R_E}\|, \|e_{\w_E}\|]^\T\in\Re^2$ and the matrices are defined as
	\begin{gather*} 
	W_1=\frac{c_1}{2}\left[\begin{array}{cc}\frac{B_3^*}{c_1} & k_\W+B_1^*\\ k_\W+B_1^* & k_R\end{array}\right],\\
	W_2=\frac{c_1}{2}\left[\begin{array}{cc}k_R &\frac{k_\W}{\lam_m} \\ \frac{k_\W}{\lam_m}&\frac{c_2A_E}{3c_1\lam_M} \end{array}\right],\\
	W_3=\frac{1}{2}\left[\begin{array}{cc} B_3^* & -\frac{k_\W}{\lam_m} \\ -\frac{k_\W}{\lam_m} & \frac{c_2A_E}{3\lam_M\lam_m}\end{array}\right],\\
	W_4=\frac{1}{2\lam_m}\left[\begin{array}{cc} k_E\frac{2k_v\lam_m-c_2\lam_M}{\lam_M} & -\frac{ck_vB_E}{2}  \\ -\frac{ck_vB_E}{2} &  \frac{c_2A_E}{3\lam_M} \end{array}\right].
	\end{gather*}
The constants $c_1,c_2$ and the controller gains can be chosen such that all of the above matrices are positive-definite. For example, if the constant $c_1$ is chosen such that
	\begin{align} 
	c_1 &\leq \min\big\{\frac{2k_\W}{\lam_M(\sqrt{2}\tr[G]+B_2^*)} ,\,\frac{c_2k_R\lam_m^2A_E}{3k_\W^2\lam_M} \no\\
	&\quad \frac{2k_Rk_\W}{k_R\lam_M(\sqrt{2}\tr[G]+B_2^*) +2(k_\W+B_1^*)^2}\big\},
	\end{align} 
then the matrices $W_1$ and $W_2$ are positive definite. Conditions on $c_2$ to guarantee the positive-definiteness of $W_3$ and $W_4$ can be derived similarly. Therefore, the zero equilibrium of the tracking errors and estimation errors is exponentially stable. 
	
%This fist condition secures that $B_3^*>0$. The second and third conditions ensure the positive-definiteness of $W_1$ and $W_2$, respectively. Expressly,  when $c_1$ is sufficiently small, $W_1$ and $W_2$ are positive definite. Next, we select the constant $c_2$ satisfying	
%	\begin{gather} 	
%	\frac{\lam_Mk_\W^2}{\lam_m A_EB_3^*} <c_2< \frac{8k_v\lam_m}{\lam_M[4+\frac{3\lam_M(k_vB_E)^2}{k_EA_E}]}, \label{eq:c2}
%	\end{gather}
%Note that the lower bound and upper bound of $c_2$ ensure the positive-definiteness of $W_3$ and $W_4$, respectively. Moreover, The constant $c_2$ is chosen such that \refeqn{c1} and $\refeqn{c2}$ are satisfied, which shows that the desired equilibrium is exponentially stable. It is worthwhile to mention that the constants $c_1$ and $c_2$ do not exist either in the observer or the control law. Hence, the value of $c_1$ and $c_2$ does not affect the performance of the control system. 

%==================================================================
%\begin{align}  \end{align} \begin{align*} \end{align*}
%\begin{gather} \end{gather} \begin{gather*} \end{gather*}
%\label{eq:XXX}(\ref{eq:XXX})
%\empg{text here} \textit{text here}
%\left[\begin{array}{c}x_1\\x_2\end{array}\right]
%\eval{\frac{dX}{dX}}_{X=0} \pd{f_1}{x_1}
%\begin{enumerate} \item \item \item \end{enumerate}
%\begin{itemize} \item \item \item \end{itemize} 
%\int_{a}^{b} f(x) \,dx

\bibliography{ECC15}
\bibliographystyle{IEEEtran}
\end{document}